\documentclass[12pt]{article}
\usepackage[]{amsmath,amssymb}
\usepackage{amscd}
\usepackage{latexsym}
\usepackage{cite}
\usepackage{amsthm}

\newtheorem{definition}{Definition}[section]
\newtheorem{theorem}[definition]{Theorem}
\newtheorem{lemma}[definition]{Lemma}
\newtheorem{corollary}[definition]{Corollary}
\newtheorem{remark}[definition]{Remark}
\newtheorem{example}[definition]{Example}
\newtheorem{conjecture}[definition]{Conjecture}
\newtheorem{problem}[definition]{Problem}
\newtheorem{note}[definition]{Note}

\newtheorem{proposition}[definition]{Proposition}

\begin{document}
\title{\bf 
The $q$-Onsager algebra and its \\alternating central extension}
 \author{
Paul Terwilliger 
}
\date{}

\maketitle
\begin{abstract} The $q$-Onsager algebra $O_q$ has a presentation involving two generators $W_0$, $W_1$ and two relations, called the
$q$-Dolan/Grady relations. The alternating central extension $\mathcal O_q$ has a presentation involving the alternating generators
$\lbrace \mathcal W_{-k}\rbrace_{k=0}^\infty$, $\lbrace \mathcal  W_{k+1}\rbrace_{k=0}^\infty$, 
$ \lbrace \mathcal G_{k+1}\rbrace_{k=0}^\infty$, 
$\lbrace \mathcal {\tilde G}_{k+1}\rbrace_{k=0}^\infty$ and a large number of relations. 
Let $\langle \mathcal W_0, \mathcal W_1 \rangle$ denote the subalgebra of $\mathcal O_q$ generated by $\mathcal W_0$, $\mathcal W_1$.
It is known that there exists an algebra isomorphism $O_q \to \langle \mathcal W_0, \mathcal W_1 \rangle$ that sends $W_0\mapsto \mathcal W_0$
and $W_1 \mapsto \mathcal W_1$.
It is known that the center $\mathcal Z$ of $\mathcal O_q$ is isomorphic to a polynomial algebra in countably many variables.
 It is known that the multiplication map 
$\langle \mathcal W_0, \mathcal W_1 \rangle \otimes \mathcal Z \to \mathcal O_q$, $ w \otimes z \mapsto wz$ is an isomorphism of algebras.
We call this isomorphism the standard tensor product factorization of $\mathcal O_q$. In the study of $\mathcal O_q$ there are two natural points of view: we can start with
the alternating generators, or we can start with the standard tensor product factorization. It is not obvious how these two points of view are related.
The goal of the paper is to
describe this relationship. We give seven main results; the principal one is an attractive factorization of the generating function for some algebraically
independent elements that generate $\mathcal Z$.
\bigskip


 \end{abstract}
 
 \section{Introduction} This paper is part of a sequence 
\cite{conj},
\cite{pbwqO}, \cite{compQons} concerning the $q$-Onsager algebra and its alternating central extension. We refer to those papers
for background information and historical remarks. Let us recall a few main points.
 The $q$-Onsager algebra $O_q$ is associative and infinite-dimensional. It has a presentation with two generators $W_0$, $W_1$ and two relations, called the $q$-Dolan/Grady relations:
 \begin{align*}
&\lbrack W_0, \lbrack W_0, \lbrack W_0, W_1\rbrack_q \rbrack_{q^{-1}} \rbrack =(q^2 - q^{-2})^2 \lbrack W_1, W_0 \rbrack,
\\
&\lbrack W_1, \lbrack W_1, \lbrack W_1, W_0\rbrack_q \rbrack_{q^{-1}}\rbrack = (q^2-q^{-2})^2 \lbrack W_0, W_1 \rbrack.
\end{align*}

\noindent 
In \cite[Theorem~4.5]{BK},
Baseilhac and Kolb obtain a Poincar\'e-Birkhoff-Witt (or PBW)
basis for $O_q$. They obtain this PBW basis by using an approach of Damiani 
\cite{damiani}
along with
two automorphisms of
$O_q$ that are reminiscent of the
Lusztig automorphisms for
$U_q(\widehat{\mathfrak{sl}}_2)$.
The PBW basis elements are denoted
\begin{align}
\lbrace B_{n \delta+ \alpha_0} \rbrace_{n=0}^\infty,
\qquad \quad 
\lbrace B_{n \delta+ \alpha_1} \rbrace_{n=0}^\infty,
\qquad \quad 
\lbrace B_{n \delta} \rbrace_{n=1}^\infty.
\label{eq:UpbwIntro}
\end{align}
 We will be discussing the generating functions
 \begin{align*}
&B^-(t) = \sum_{n=0}^\infty  B_{n\delta+\alpha_0} t^n, \qquad \quad
B^+(t) = \sum_{n=0}^\infty  B_{n\delta+\alpha_1} t^n,
\\
&B(t) = \sum_{n=0}^\infty B_{n\delta} t^n,       \qquad \qquad B_{0\delta} = q^{-2}-1.
\end{align*}

\noindent
 In \cite{BK05} Baseilhac and Koizumi introduce a current algebra  for $O_q$, in order to solve boundary integrable systems with hidden symmetries.  Following
 \cite{compQons} we denote this current algebra by $\mathcal O_q$.
In \cite[Definition~3.1]{basnc} Baseilhac and Shigechi give a presentation of $\mathcal O_q$ by generators and relations. 
The generators, said to be alternating, are denoted
\begin{align*}
\lbrace \mathcal W_{-k}\rbrace_{k=0}^\infty, \qquad  \lbrace \mathcal  W_{k+1}\rbrace_{k=0}^\infty,\qquad  
 \lbrace \mathcal G_{k+1}\rbrace_{k=0}^\infty,
\qquad
\lbrace \mathcal {\tilde G}_{k+1}\rbrace_{k=0}^\infty.
\end{align*}
 The relations are given in  \eqref{eq:3p1}--\eqref{eq:3p11}
 below.  The alternating generators form a PBW basis for $\mathcal O_q$, see \cite[Theorem~6.1]{pbwqO} or Proposition \ref{lem:pbw} below.
 We will be discussing the generating functions
  \begin{align*}
&\mathcal W^-(t) = \sum_{n=0}^\infty \mathcal W_{-n} t^n,
\qquad \qquad  
\mathcal W^+(t) = \sum_{n=0}^\infty \mathcal W_{n+1} t^n,
\\
&\mathcal G(t) = \sum_{n=0}^\infty \mathcal G_n t^n,
\qquad \quad 
\mathcal {\tilde G}(t) = \sum_{n=0}^\infty \mathcal {\tilde G}_n t^n, \qquad \quad \mathcal G_0 = \mathcal {\tilde G}_0 = - (q-q^{-1})\lbrack 2 \rbrack^2_q.
\end{align*}

 \noindent Next we describe how $O_q$ and $\mathcal O_q$ are related. In this description we will refer to the   polynomial algebra $ \mathbb F \lbrack z_1, z_2, \ldots \rbrack$,
where $\mathbb F$ is the ground field and $\lbrace z_n \rbrace_{n=1}^\infty$ are mutually commuting indeterminates. For notational convenience define $z_0=1$.
 \medskip

 \noindent 
 The algebras $O_q$ and $\mathcal O_q$ are related as follows.
 \begin{enumerate}
 \item[$\bullet$]
 Let $\langle \mathcal W_0, \mathcal W_1\rangle $ denote the subalgebra of $\mathcal O_q$ generated by $\mathcal W_0$, $\mathcal W_1$.
By \cite[Theorem~10.3]{pbwqO} there exists an algebra isomorphism $O_q \to  \langle \mathcal W_0, \mathcal W_1\rangle $ that sends $W_0 \mapsto \mathcal W_0$ and
$W_1 \mapsto \mathcal W_1$.
 \item[$\bullet$]  By \cite[Theorem~10.2]{pbwqO}  the center $\mathcal Z$ of $\mathcal O_q$ is isomorphic to $ \mathbb F \lbrack z_1, z_2, \ldots \rbrack$.
\item[$\bullet$] By \cite[Theorem~10.4]{pbwqO} the multiplication map 
\begin{align}
\begin{split}
\langle \mathcal W_0, \mathcal W_1\rangle 
\otimes
\mathcal Z
 & \to   \mathcal O_q
\\
 w \otimes z  &\mapsto  wz
 \end{split}
 \label{eq:mult}
 \end{align}
 is an isomorphism of algebras.
\end{enumerate}
By the above bullets points or  \cite[Theorem~9.14]{pbwqO},  the algebra
 $\mathcal O_q$ is isomorphic to  $O_q\otimes  \mathbb F \lbrack z_1, z_2, \ldots \rbrack$.
Motivated by this and as explained in \cite{pbwqO}, 
we call $\mathcal O_q$  the alternating central extension of $O_q$.  We call \eqref{eq:mult}  the standard tensor product factorization of $\mathcal O_q$.
 \medskip

  \noindent For the rest of this section, we identify the algebra $O_q$ with  $\langle \mathcal W_0, \mathcal W_1\rangle $ via the isomorphism in the first bullet point above.
 \medskip
 
 \noindent In the study of $\mathcal O_q$ there are two natural points of view: we can start with the 
 alternating generators, or we can start with the standard tensor product factorization. It is not obvious how these two
 points of view are related. The goal of the paper is to describe this relationship.
 \medskip
 
 \noindent In this paper we obtain seven main results, that are summarized as follows.
 \medskip

 \noindent
  Our first and second main results relate the generating functions for $O_q$ and $\mathcal O_q$. We show that
\begin{align*}
\frac{q+q^{-1}}{t+t^{-1}}
\mathcal W^-\biggl( \frac{q+q^{-1}}{t+t^{-1}}\biggr) &=
\frac{q^{-1}t B^+(q^{-1}t)+ B^-(qt)    
}{(q^2-q^{-2})(t-t^{-1})}\,
\mathcal {\tilde G}\biggl( \frac{q+q^{-1}}{t+t^{-1}}\biggr),
\\
\frac{q+q^{-1}}{t+t^{-1}}
\mathcal W^+\biggl( \frac{q+q^{-1}}{t+t^{-1}}\biggr) &=
\frac{ B^+(q^{-1}t)+ qtB^-(qt) }{(q^2-q^{-2})(t-t^{-1})}\,
\mathcal {\tilde G}\biggl( \frac{q+q^{-1}}{t+t^{-1}}\biggr)
\end{align*}
and also 
\begin{align*}
\frac{q+q^{-1}}{t+t^{-1}}
\mathcal W^-\biggl( \frac{q+q^{-1}}{t+t^{-1}}\biggr) &=
\mathcal {\tilde G}\biggl( \frac{q+q^{-1}}{t+t^{-1}}\biggr)
\frac{qt B^+(qt)+ B^-(q^{-1}t)    
}{(q^2-q^{-2})(t-t^{-1})},
\\
\frac{q+q^{-1}}{t+t^{-1}}
\mathcal W^+\biggl( \frac{q+q^{-1}}{t+t^{-1}}\biggr) &=
\mathcal {\tilde G}\biggl( \frac{q+q^{-1}}{t+t^{-1}}\biggr)
\frac{ B^+(qt)+ q^{-1}tB^-(q^{-1}t) }{(q^2-q^{-2})(t-t^{-1})}.
\end{align*}

 \noindent Our third main result concerns some elements $\lbrace \mathcal Z^\vee_n \rbrace_{n=1}^\infty$ from \cite[Theorem~10.2]{pbwqO}
that are algebraically independent and generate $\mathcal Z$.
 By \cite[Definitions~8.4,~8.6]{pbwqO}, the generating function
  \begin{align*}
  \mathcal Z^\vee(t) = \sum_{n=0}^\infty \mathcal Z^\vee_n t^n, \qquad \qquad \mathcal Z^\vee_0=1
 \end{align*}
 satisfies $\mathcal Z^\vee(t) = (q+q^{-1})^{-2} \Psi(t)$, where
 \begin{align*} 
 \Psi(t)&=
t^{-1} ST\mathcal W^-(S) \mathcal W^+(T) +
t ST\mathcal W^+(S) \mathcal W^-(T) -
q^2 ST \mathcal W^-(S) \mathcal W^-(T)
\\
& \quad \qquad -
q^{-2}ST \mathcal W^+(S) \mathcal W^+(T) +
(q^2-q^{-2})^{-2} \mathcal G(S) \mathcal {\tilde G}(T)
\end{align*}
and 
\begin{align}\label{eq:IntroST}
 S = \frac{q+q^{-1}}{q^{-1}t+qt^{-1}}, \qquad \quad  T= \frac{q+q^{-1}}{qt+q^{-1}t^{-1}}.
 \end{align}
 We obtain a factorization
 \begin{align}
 \label{eq:factor}
 \mathcal Z^\vee(t) = \xi \mathcal {\tilde G}(S) B(t) \mathcal {\tilde G}(T), \qquad \qquad \xi = -q(q-q^{-1})(q^2-q^{-2})^{-4}.
 \end{align}

 \noindent In our fourth main result, we introduce the alternating generators for $O_q$.
 These generators are denoted
 \begin{align*}
\lbrace W_{-k}\rbrace_{k=0}^\infty, \qquad  \lbrace W_{k+1}\rbrace_{k=0}^\infty,\qquad  
 \lbrace G_{k+1}\rbrace_{k=0}^\infty,
\qquad
\lbrace {\tilde G}_{k+1}\rbrace_{k=0}^\infty
\end{align*}
 and defined as follows.
 We display a surjective algebra homomorphism $\gamma: \mathcal O_q \to O_q$ that sends 
 \begin{align*}                     
 \mathcal W_0 \mapsto W_0,
 \qquad \quad
  \mathcal W_1 \mapsto W_1,
 \qquad \quad
 \mathcal Z^\vee_n \mapsto 0, \qquad n \geq 1.
 \end{align*}
 The map $\gamma$ sends
 \begin{align*}
 \mathcal W_{-k} \mapsto W_{-k}, \qquad \quad
  \mathcal W_{k+1} \mapsto W_{k+1}, \qquad \quad
   \mathcal G_{k+1} \mapsto G_{k+1}, \qquad \quad
    \mathcal {\tilde G}_{k+1} \mapsto {\tilde G}_{k+1}
 \end{align*}
 for  $k \in \mathbb N$. The existence of $\gamma$ was previously conjectured by Baseilhac and Belliard in \cite[Conjecture~2]{basBel}.
 \medskip
 
 \noindent In our fifth main result,  we prove \cite[Conjecture~6.2]{conj}. The essential ingredients in our proof are the
  factorization  \eqref{eq:factor}
  and the existence of the alternating generators of $O_q$.
\medskip
 
 \noindent In our sixth main result,  we establish an algebra isomorphism $\varphi: 
\mathcal O_q \to O_q \otimes 
\mathbb F \lbrack z_1, z_2,\ldots\rbrack$ that sends
\begin{align*}
\mathcal W_{-n} &\mapsto \sum_{k=0}^n W_{k-n} \otimes z_k,
\quad \qquad \qquad 
\mathcal W_{n+1} \mapsto \sum_{k=0}^n W_{n+1-k} \otimes z_k,
\\
\mathcal G_{n} &\mapsto \sum_{k=0}^n G_{n-k} \otimes z_k,
\quad \qquad \qquad
\mathcal {\tilde G}_{n} \mapsto \sum_{k=0}^n \tilde G_{n-k} \otimes z_k
\end{align*}
for $n \in \mathbb N$. In particular $\varphi$ sends
\begin{align*}
\mathcal W_0 \mapsto W_0 \otimes 1,
\qquad \qquad 
\mathcal W_1 \mapsto W_1 \otimes 1.
\end{align*}

\noindent For $n \in \mathbb N$ let $\mathcal Z_n$ denote the preimage of $1\otimes z_n$ under $\varphi$. We have $\mathcal Z_0=1$. By construction the elements
$\lbrace \mathcal Z_n \rbrace_{n=1}^\infty$
  are  algebraically independent and generate $\mathcal Z$.
  \medskip
  
  \noindent In our seventh main result, we show that the elements $\lbrace \mathcal Z_n \rbrace_{n=1}^\infty$ are related to the elements $\lbrace \mathcal Z^\vee_n \rbrace_{n=1}^\infty$
  in the following way. The generating function 
 $\mathcal Z(t) = \sum_{n=0}^\infty \mathcal Z_n t^n$ satisfies
  \begin{align*}
  \mathcal Z^\vee(t) = \mathcal Z(S)\mathcal Z(T),
  \end{align*}
  where $S$, $T$ are from \eqref{eq:IntroST}.
 \medskip

\noindent Our seven main results are contained in Theorems
 \ref{lem:zzznote}, 
   \ref{lem:zzznotez},
     \ref{thm:Mn},
      \ref{lem:4gen},
     \ref{thm:ct},
   \ref{prop:vpIso},
    \ref{lem:ZST}.
    \medskip
    
    \noindent
    We remark that \cite{alternating}, \cite{altCE} contain analogs of the above results that apply to the alternating central extension for
    the positive part of $U_q(\widehat{\mathfrak{sl}}_2)$; see also \cite{compactUqp}.
    \medskip
    
     \noindent At the end of the paper we give some conjectures and open problems.
     \medskip
   
  \noindent The paper is organized as follows. Section 2 contains some preliminaries. 
  In Section 3 we give the definition and basic properties of $O_q$.
  In Sections 4, 5 we describe the PBW basis for $O_q$ due to Baseilhac and Kolb.
    In Sections 6, 7 we give the definition and basic properties of $\mathcal O_q$.
In Section 8 we recall an algebra isomorphism $\phi : O_q \otimes \mathbb F \lbrack z_1, z_2, \ldots \rbrack\to \mathcal O_q$
and the generating function $\mathcal Z^\vee(t)$.
In Section 9 we compare our generating functions for $O_q$ and $\mathcal O_q$.
In Section 10 we obtain our factorization of $\mathcal Z^\vee(t)$.
In Section 11 we introduce the algebra homomorphism $\gamma: \mathcal O_q \to O_q$
and use it to obtain the alternating generators of $O_q$.
In Section 12 we use the factorization of $\mathcal Z^\vee(t)$ and the alternating generators of $O_q$ to
prove \cite[Conjecture~6.2]{conj}.
In Section 13 we obtain the algebra isomorphism $\varphi: \mathcal O_q \to  O_q \otimes \mathbb F \lbrack z_1, z_2, \ldots \rbrack$
and describe how it is related to the inverse of $\phi$. We also obtain the generating function $\mathcal Z(t)$ and describe
how it is related to $\mathcal Z^\vee(t)$. 
In Sections 14, 15 we tie up some loose ends from earlier sections.
In Section 16 we give some conjectures and open problems.
In Appendix A we describe some features of the polynomial algebra $\mathbb F \lbrack z_1, z_2, \ldots \rbrack$ that are
used in the main body of the paper.

\section{Preliminaries}
\noindent We now begin our formal argument. Throughout the paper, the following notational conventions are in effect.
Recall the natural numbers $\mathbb N= \lbrace 0,1,2,\ldots \rbrace$ and integers $\mathbb Z=\lbrace 0,\pm 1, \pm 2,\ldots \rbrace$. Let $\mathbb F$ denote a field.
Every vector space and tensor product discussed in this paper is over $\mathbb F$.
Every algebra discussed in this paper is associative, over $\mathbb F$, and has a multiplicative identity. A subalgebra has the same multiplicative identity as the parent algebra.
 Let $\mathcal A$ denote an algebra. By an {\it automorphism} of $\mathcal A$
we mean an algebra isomorphism $\mathcal A\rightarrow \mathcal A$. The algebra $\mathcal A^{\rm opp}$ consists of the vector space $\mathcal A$ and the multiplication map $\mathcal A \times \mathcal A \rightarrow \mathcal A$, $(a,b)\to ba$.
By an {\it antiautomorphism} of $\mathcal A$ we mean an algebra isomorphism $\mathcal A \rightarrow \mathcal A^{\rm opp}$.
\medskip

\noindent Throughout the paper, let $s$ and $t$ denote commuting indeterminates.

 \begin{definition}\label{def:pbw}
 \rm 
(See \cite[p.~299]{damiani}.)
Let $ \mathcal A$ denote an algebra. A {\it Poincar\'e-Birkhoff-Witt} (or {\it PBW}) basis for $\mathcal A$
consists of a subset $\Omega \subseteq \mathcal A$ and a linear order $<$ on $\Omega$
such that the following is a basis for the vector space $\mathcal A$:
\begin{align*}
a_1 a_2 \cdots a_n \qquad n \in \mathbb N, \qquad a_1, a_2, \ldots, a_n \in \Omega, \qquad
a_1 \leq a_2 \leq \cdots \leq a_n.
\end{align*}
We interpret the empty product as the multiplicative identity in $\mathcal A$.
\end{definition}

\begin{definition}\label{def:poly} \rm
Let $\lbrace z_n \rbrace_{n=1}^\infty$ denote mutually commuting indeterminates. Let $\mathbb F \lbrack z_1, z_2, \ldots \rbrack$ denote
the algebra consisting of the polynomials in $z_1, z_2, \ldots $ that have all coefficients in $\mathbb F$.
For notational convenience, define $z_0=1$.
\end{definition}

\noindent Some features of  $\mathbb F \lbrack z_1, z_2, \ldots \rbrack$ are explained in Appendix A.
\medskip

  \noindent Throughout the paper, we fix a nonzero $q \in \mathbb F$
that is not a root of unity.
Recall the notation
\begin{align*}
\lbrack n\rbrack_q = \frac{q^n-q^{-n}}{q-q^{-1}}
\qquad \qquad n \in \mathbb N.
\end{align*}

\section{The $q$-Onsager algebra $O_q$}
In this section we recall the $q$-Onsager algebra $ O_q$.
\medskip

\noindent For elements $X, Y$ in any algebra, define their
commutator and $q$-commutator by 
\begin{align*}
\lbrack X, Y \rbrack = XY-YX, \qquad \qquad
\lbrack X, Y \rbrack_q = q XY- q^{-1}YX.
\end{align*}
\noindent Note that 
\begin{align*}
\lbrack X, \lbrack X, \lbrack X, Y\rbrack_q \rbrack_{q^{-1}} \rbrack
= 
X^3Y-\lbrack 3\rbrack_q X^2YX+ 
\lbrack 3\rbrack_q XYX^2 -YX^3.
\end{align*}

\begin{definition} \label{def:U} \rm
(See \cite[Section~2]{bas1}, \cite[Definition~3.9]{qSerre}.)
Define the algebra $O_q$ by generators $W_0$, $W_1$ and relations
\begin{align}
\label{eq:qOns1}
&\lbrack W_0, \lbrack W_0, \lbrack W_0, W_1\rbrack_q \rbrack_{q^{-1}} \rbrack =(q^2 - q^{-2})^2 \lbrack W_1, W_0 \rbrack,
\\
\label{eq:qOns2}
&\lbrack W_1, \lbrack W_1, \lbrack W_1, W_0\rbrack_q \rbrack_{q^{-1}}\rbrack = (q^2-q^{-2})^2 \lbrack W_0, W_1 \rbrack.
\end{align}
We call $O_q$ the {\it $q$-Onsager algebra}.
The relations \eqref{eq:qOns1}, \eqref{eq:qOns2}  are called the {\it $q$-Dolan/Grady relations}.
\end{definition}
\begin{remark}\rm In \cite{BK} Baseilhac and Kolb define the $q$-Onsager algebra in a slightly more general way that involves two scalar parameters $c, q$. Our $O_q$ is their
$q$-Onsager algebra with $c=q^{-1}(q-q^{-1})^2$.
\end{remark}
\noindent We mention some symmetries of $O_q$. 

\begin{lemma}
\label{lem:aut} There exists an automorphism $\sigma$ of $O_q$ that sends $W_0 \leftrightarrow W_1$.
Moreover $\sigma^2 = {\rm id}$, where ${\rm id}$ denotes the identity map.
\end{lemma}

\begin{lemma}\label{lem:antiaut} {\rm (See \cite[Lemma~2.5]{z2z2z2}.)}
There exists an antiautomorphism $\dagger$ of $O_q$ that fixes each of $W_0$, $W_1$.
 Moreover $\dagger^2={\rm id}$.
\end{lemma}

\begin{lemma} {\rm (See \cite[Lemma~3.5]{compQons}.)}  The maps $\sigma$, $\dagger$ commute.
\end{lemma}
 
\begin{definition}\label{def:tauA} \rm Let $\tau$ denote the composition of $\sigma$ and $\dagger$. Note that $\tau$ is an antiautomorphism of $O_q$ that sends
$W_0 \leftrightarrow W_1$. We have $\tau^2 = {\rm id}$.
\end{definition}

\noindent Later in the paper we will make use of the following map.
\begin{lemma}\label{lem:vth}
 There exists an algebra homomorphism $\vartheta: O_q \to \mathbb F$
that sends
\begin{align*}
W_0 \mapsto 0, \qquad \qquad W_1 \mapsto 0.
\end{align*}
\end{lemma}
\begin{proof} The $q$-Dolan/Grady relations \eqref{eq:qOns1}, \eqref{eq:qOns2} hold if we set $W_0=0$ and $W_1=0$.
\end{proof}

\section{A PBW basis for $O_q$}

\noindent In \cite{BK}, Baseilhac and Kolb obtain a PBW basis for $O_q$ that involves some elements
\begin{align}
\lbrace B_{n \delta+ \alpha_0} \rbrace_{n=0}^\infty,
\qquad \quad 
\lbrace B_{n \delta+ \alpha_1} \rbrace_{n=0}^\infty,
\qquad \quad 
\lbrace B_{n \delta} \rbrace_{n=1}^\infty.
\label{eq:Upbw}
\end{align}
These elements are recursively defined  as follows. Writing $B_\delta  = q^{-2}W_1 W_0 - W_0 W_1$ we have
\begin{align}
&B_{\alpha_0}=W_0,  \qquad \qquad 
B_{\delta+\alpha_0} = W_1 + 
\frac{q \lbrack B_{\delta}, W_0\rbrack}{(q-q^{-1})(q^2-q^{-2})},
\label{eq:line1}
\\
&
B_{n \delta+\alpha_0} = B_{(n-2)\delta+\alpha_0}
+ 
\frac{q \lbrack B_{\delta}, B_{(n-1)\delta+\alpha_0}\rbrack}{(q-q^{-1})(q^2-q^{-2})} \qquad \qquad n\geq 2
\label{eq:line2}
\end{align}
and 
\begin{align}
&B_{\alpha_1}=W_1,  \qquad \qquad 
B_{\delta+\alpha_1} = W_0 - 
\frac{q \lbrack B_{\delta}, W_1\rbrack}{(q-q^{-1})(q^2-q^{-2})},
\label{eq:line3}
\\
&
B_{n \delta+\alpha_1} = B_{(n-2)\delta+\alpha_1}
- 
\frac{q \lbrack B_{\delta}, B_{(n-1)\delta+\alpha_1}\rbrack}{(q-q^{-1})(q^2-q^{-2})} \qquad \qquad n\geq 2.
\label{eq:line4}
\end{align}
Moreover for $n\geq 1$,
\begin{equation}
\label{eq:Bdelta}
B_{n \delta} = 
q^{-2}  B_{(n-1)\delta+\alpha_1} W_0
- W_0 B_{(n-1)\delta+\alpha_1}  + 
(q^{-2}-1)\sum_{\ell=0}^{n-2} B_{\ell \delta+\alpha_1}
B_{(n-\ell-2) \delta+\alpha_1}.
\end{equation}
By \cite[Proposition~5.12]{BK} the elements $\lbrace B_{n\delta}\rbrace_{n=1}^\infty$ mutually commute.

\begin{lemma}
\label{prop:damiani} 
{\rm (See \cite[Theorem~4.5]{BK}.)}
Assume that $q$ is transcendental over $\mathbb F$. Then
 a PBW basis for $O_q$ is obtained by the elements {\rm \eqref{eq:Upbw}} in any linear
order.
\end{lemma}

\noindent 
 We mention a variation on the formula \eqref{eq:Bdelta}.
By \cite[Section~5.2]{BK} the following
holds  for $n\geq 1$:
\begin{equation}
\label{eq:Bdel2}
B_{n \delta} = 
q^{-2} W_1 B_{(n-1)\delta+\alpha_0} 
-  B_{(n-1)\delta+\alpha_0}  W_1 + 
(q^{-2}-1)\sum_{\ell=0}^{n-2} B_{\ell \delta+\alpha_0}
B_{(n-\ell-2) \delta+\alpha_0}.
\end{equation}
\noindent 
Recall the antiautomorphism $\tau$ of $O_q$, from Definition \ref{def:tauA}.
\begin{lemma} \label{lem:asym2} The antiautomorphism $\tau$  sends
$B_{n\delta+\alpha_0}\leftrightarrow B_{n\delta+\alpha_1}$ for $n \in \mathbb N$, and fixes $B_{n\delta}$ for $n\geq 1$.
\end{lemma}
\begin{proof} The first assertion is verified by comparing \eqref{eq:line1}, \eqref{eq:line2} with \eqref{eq:line3}, \eqref{eq:line4}. The second assertion is verified by
comparing \eqref{eq:Bdelta}, \eqref{eq:Bdel2}.
\end{proof}

\begin{lemma}\label{lem:vth2}
The algebra homomorphism $\vartheta: O_q \to \mathbb F$ from Lemma \ref{lem:vth} sends
$B_{n \delta + \alpha_0} \mapsto 0$ and $B_{n \delta + \alpha_1} \mapsto 0$ for $n \in \mathbb N$,  and
$B_{n \delta} \mapsto 0$ for $n\geq 1$.
\end{lemma}
\begin{proof} The map $\vartheta$ sends $B_\delta \mapsto 0$ by Lemma \ref{lem:vth} and since $B_\delta = q^{-2} W_1W_0-W_0W_1$. The remaining assertions follow from \eqref{eq:line1}--\eqref{eq:Bdelta}.
\end{proof}

\section{Generating functions for $O_q$}

\noindent In this section, we describe the elements \eqref{eq:Upbw} using generating functions.
\medskip

\noindent The following definition is for notational convenience.
\begin{definition}\label{def:Bnd}\rm  For an integer $n\leq 0$ define
\begin{align*}
B_{n\delta} = \begin{cases}
0, & {\mbox{\rm if $n<0$}}; \\
q^{-2}-1, &{\mbox{\rm if $n=0$.}}
\end{cases}
\end{align*}
\end{definition}

\begin{definition} \label{def:Bgen}
\rm We define some generating functions in the indeterminate $t$:
\begin{align}
&B^-(t) = \sum_{n  \in \mathbb N}  B_{n\delta+\alpha_0} t^n, \qquad \quad
B^+(t) = \sum_{n  \in \mathbb N}  B_{n\delta+\alpha_1} t^n,
\\
&B(t) = \sum_{n\in \mathbb N} B_{n\delta} t^n. 
 \label{eq:zerodelta}
 \end{align}
\end{definition}
\noindent Observe that
\begin{align}
B^-(0)=W_0, \qquad \qquad B^+(0)=W_1, \qquad \qquad B(0)=q^{-2}-1.
\label{eq:zv}
\end{align}

\begin{lemma}
\label{lem:BPhi}
For the algebra $O_q$,
\begin{align}
\label{eq:BP1}
&\frac{q \lbrack B_\delta, B^-(t) \rbrack}{(q-q^{-1})(q^2-q^{-2})}= 
(t^{-1}-t)B^-(t)-t^{-1} W_0 -W_1,
\\
\label{eq:BP2}
&\frac{q \lbrack B^+(t), B_\delta\rbrack}{(q-q^{-1})(q^2-q^{-2})}= 
(t^{-1}-t)B^+(t)-W_0-t^{-1} W_1.
\end{align}
\end{lemma}
\begin{proof} Equation \eqref{eq:BP1} expresses \eqref{eq:line1}, \eqref{eq:line2}
in terms of generating functions.
Equation \eqref{eq:BP2} expresses \eqref{eq:line3}, \eqref{eq:line4}
in terms of generating functions.
\end{proof}

  \begin{lemma} \label{prop:pbwRelP} 
  For the algebra $O_q$,
  \begin{align}
  \lbrack W_0, B^+(t) \rbrack_q &= 
  -(q-q^{-1}) t \bigl( B^+(t )\bigr)^2
  - q t^{-1} B(t) -(q-q^{-1})t^{-1},        \label{eq:L1}
  \\
  \lbrack B^-( t), W_1 \rbrack_q &= 
  -(q-q^{-1}) t \bigl( B^-(t )\bigr)^2
  -  qt^{-1} B(t) -(q-q^{-1})t^{-1}.         \label{eq:L2}
  \end{align}
  \end{lemma}
  \begin{proof} The relation \eqref{eq:L1} (resp. \eqref{eq:L2}) expresses the relation \eqref{eq:Bdelta} (resp. \eqref{eq:Bdel2}) in terms of generating
  functions.
  \end{proof}

\noindent The following notation will be useful.
For an integer $k<0$ define
\begin{align*}
B_{k \delta + \alpha_0} = B_{(-k-1)\delta + \alpha_1}, \qquad \qquad 
B_{k\delta+\alpha_1} = B_{(-k-1)\delta+\alpha_0}.
\end{align*}
Note that
\begin{align}
 B_{k\delta+\alpha_0} = B_{\ell\delta+\alpha_1} \qquad \qquad (k,\ell \in \mathbb Z, \quad k+\ell=-1).
 \label{eq:extend}
 \end{align}

  \begin{proposition} \label{prop:wang} {\rm (See \cite[Proposition~2.6]{LuWang}.)}
 For $k,\ell \in \mathbb Z$ we have
  \begin{align*}
  &q^{-1} \lbrack B_{\ell\delta+\alpha_1}, B_{(k+1)\delta+\alpha_1} \rbrack_q
  +
  q^{-1} \lbrack B_{k\delta+\alpha_1}, B_{(\ell+1)\delta+\alpha_1} \rbrack_q
  \\
  & \qquad = 
  B_{(k-\ell-1)\delta} - B_{(k-\ell+1)\delta} + B_{(\ell-k-1)\delta}-B_{(\ell-k+1)\delta}.
  \end{align*}
   \end{proposition}
   \begin{proof} This follows from \cite[Proposition~2.6]{LuWang} using Remark \ref{rem:LW} below.
 \end{proof}

  \begin{proposition} \label{prop:wang2} {\rm (See \cite[Proposition~2.8]{LuWang}.)}
   For $m \in \mathbb N$ and $\ell \in \mathbb Z$,
  \begin{align*}
  \lbrack B_{(m+1)\delta }, B_{\ell \delta+\alpha_1} \rbrack
  -
  \lbrack B_{\ell \delta+\alpha_1},  B_{(m-1)\delta } \rbrack
  =
  \lbrack B_{m\delta }, B_{(\ell+1) \delta+\alpha_1} \rbrack_{q^2} 
  -
  \lbrack  B_{(\ell-1)\delta +\alpha_1 }, B_{m\delta} \rbrack_{q^2}.
  \end{align*}
  \end{proposition}
  \begin{proof}
  This follows from \cite[Proposition~2.8]{LuWang} using Remark \ref{rem:LW} below.
\end{proof}

\begin{remark} \label{rem:LW}\rm
The notation of \cite{LuWang} corresponds to ours in the following way.
\bigskip

\centerline{
\begin{tabular}[t]{cc}
 {\rm our notation} & {\rm notation of \cite{LuWang}}
   \\
\hline
$q$ & $v$
\\
$W_0$ & $v^{1/2} (v-v^{-1})B_0$
\\
$W_1$ & $v^{1/2}(v-v^{-1})B_1$
\\
$B_{n\delta+\alpha_0} $ & $ v^{1/2}(v-v^{-1}) B_{-1,n}$
\\
$B_{n\delta+\alpha_1} $ & $ v^{1/2}(v-v^{-1}) B_{1,n}$
\\
$B_{n\delta} $ & $ -v^{-1}(v-v^{-1})^2 \Theta'_n$
\\
$1$, $1 $ & $\mathbb K_{n\delta+\alpha_1}$, $\mathbb K_\delta$
\\
$\lbrack X, Y\rbrack_q$ & $v \lbrack X,Y\rbrack_{v^{-2}}$
	       \end{tabular}}

\end{remark}

  \noindent Recall the commuting indeterminates $s$, $t$. By the comment above Lemma \ref{prop:damiani}, 
  \begin{align*}
  \lbrack B(s), B(t)\rbrack=0.
  \end{align*}
  
   \begin{proposition}\label{prop:GFwang} For the algebra $O_q$,
  \begin{align}
  \begin{split}
  0&=  (qs-q^{-1} t)B^+(s) B^+(t) +(qt-q^{-1}s)B^+(t) B^+(s)
  -(q-q^{-1}) t \bigl( B^+(t)\bigr)^2
  \\
  &\quad -(q-q^{-1}) s \bigl( B^+(s)\bigr)^2
  + \frac{q(s-t)}{1-st} B(t) + \frac{q(t-s)}{1-st} B(s),
  \end{split} \label{eq:cc1}
  \\
  \begin{split}
  0 &= (1-q^{-2} s t)B^-(s) B^+(t) + (st-q^{-2})B^+(t) B^-(s)
  +(1-q^{-2}) t\bigl(B^+(t)\bigr)^2 
  \\
  &\quad + (1-q^{-2}) s\bigl(B^-(s)\bigr)^2  + \frac{1-st}{s-t} B(s) - \frac{1-st}{s-t} B(t),
  \end{split}\label{eq:cc2}
  \\
  \begin{split}
  0&=(qs-q^{-1} t)B^-(t) B^-(s) +(qt-q^{-1}s)B^-(s) B^-(t)
  -(q-q^{-1}) t \bigl( B^-(t)\bigr)^2
  \\
  &\quad -(q-q^{-1}) s \bigl( B^-(s)\bigr)^2
  + \frac{q(s-t)}{1-st} B(t) + \frac{q(t-s)}{1-st} B(s).
  \end{split} \label{eq:cc3}
  \end{align}
  \end{proposition}
  \begin{proof} We first obtain \eqref{eq:cc1}. Observe that 
  \begin{align}
   \sum_{k=-1}^\infty \sum_{\ell=-1}^\infty
   \lbrack B_{\ell\delta+\alpha_1}, B_{(k+1)\delta+\alpha_1} \rbrack_q
   \label{eq:c1}
    s^k t^\ell &= 
   \bigl\lbrack t^{-1} W_0 + B^+(t), s^{-1} B^+(s) \bigr\rbrack_q,
   \\
      \label{eq:c2}
      \sum_{k=-1}^\infty \sum_{\ell=-1}^\infty
   \lbrack B_{k\delta+\alpha_1}, B_{(\ell+1)\delta+\alpha_1} \rbrack_q
       s^k t^\ell &=
     \bigl\lbrack s^{-1} W_0 + B^+(s), t^{-1} B^+(t) \bigr\rbrack_q
     \end{align}
     and also
     \begin{align}
   \label{eq:c3}
    \sum_{k=-1}^\infty \sum_{\ell=-1}^\infty
  B_{(k-\ell-1)\delta} 
        s^k t^\ell &=B(s) \frac{1}{t} \,\frac{1}{1-st},
        \\
           \label{eq:c4}
      \sum_{k=-1}^\infty \sum_{\ell=-1}^\infty
   B_{(k-\ell+1)\delta}      s^k t^\ell &= B(s) \frac{ 1}{s^2 t} \, \frac{1}{1-st} +\frac{1-q^{-2}}{s^2t},
  \\
     \label{eq:c5}
      \sum_{k=-1}^\infty \sum_{\ell=-1}^\infty
  B_{(\ell-k-1)\delta}    s^k t^\ell &=B(t) \frac{1}{s} \,\frac{1}{1-st},
  \\
     \label{eq:c6}
      \sum_{k=-1}^\infty \sum_{\ell=-1}^\infty
  B_{(\ell-k+1)\delta}
       s^k t^\ell &= B(t) \frac{ 1}{s t^2} \, \frac{1}{1-st} +\frac{1-q^{-2}}{st^2}.
  \end{align}
 Let $E$ denote the equation that is  
  $q^{-1}$ times \eqref{eq:c1} 
  plus
    $q^{-1}$ times \eqref{eq:c2} 
    minus  \eqref{eq:c3} 
        plus \eqref{eq:c4} 
         minus \eqref{eq:c5} 
           plus \eqref{eq:c6}. The left-hand side of $E$ is equal to zero, by   Proposition \ref{prop:wang}. In the right-hand side of $E$,  eliminate $\lbrack W_0, B^+(s)\rbrack_q$ and  $\lbrack W_0, B^+(t)\rbrack_q$ using   \eqref{eq:L1},
        and simplify the result. This yields \eqref{eq:cc1}. 
           Next we obtain \eqref{eq:cc2}.
           We have a preliminary comment about Proposition \ref{prop:wang}. In that proposition, replace $k$ by $-k-2$ and use  \eqref{eq:extend} to obtain
  \begin{align*}
  &q^{-1} \lbrack B_{\ell\delta+\alpha_1}, B_{k\delta+\alpha_0} \rbrack_q
  +
  q^{-1} \lbrack B_{(k+1)\delta+\alpha_0}, B_{(\ell+1)\delta+\alpha_1} \rbrack_q
  \\
  & \qquad = 
  B_{-(k+\ell+3)\delta} - B_{-(k+\ell+1)\delta} + B_{(k+\ell+1)\delta}-B_{(k+\ell+3)\delta}
  \end{align*}
for $k,\ell \in \mathbb Z$. We are done with our preliminary comment. Observe that
  \begin{align}
   \sum_{k=0}^\infty \sum_{\ell=0}^\infty
   \lbrack B_{\ell\delta+\alpha_1}, B_{k\delta+\alpha_0} \rbrack_q
   \label{eq:2c1}
    s^k t^\ell &= 
   \bigl\lbrack B^+(t),  B^-(s) \bigr\rbrack_q,
   \\
      \label{eq:2c2}
      \sum_{k=0}^\infty \sum_{\ell=0}^\infty
   \lbrack B_{(k+1)\delta+\alpha_0}, B_{(\ell+1)\delta+\alpha_1} \rbrack_q
       s^k t^\ell &=
    s^{-1}t^{-1} \bigl\lbrack  B^-(s)-W_0, B^+(t)-W_1 \bigr\rbrack_q
     \end{align}
     and also
     \begin{align}
   \label{eq:2c3}
    \sum_{k=0}^\infty \sum_{\ell=0}^\infty
  B_{-(k+\ell+3)\delta} 
        s^k t^\ell &=0,
        \\
           \label{eq:2c4}
      \sum_{k=0}^\infty \sum_{\ell=0}^\infty
   B_{-(k+\ell+1)\delta}      s^k t^\ell &=0,
  \\
     \label{eq:2c5}
      \sum_{k=0}^\infty \sum_{\ell=0}^\infty
  B_{(k+\ell+1)\delta}    s^k t^\ell &=\frac{B(s)-B(t)}{s-t},
  \\
     \label{eq:2c6}
      \sum_{k=0}^\infty \sum_{\ell=0}^\infty
  B_{(k+\ell+3)\delta}
       s^k t^\ell &= \frac{s^{-2}(B(s)-s B_\delta+1-q^{-2})-t^{-2}(B(t)-tB_\delta+1-q^{-2})}{s-t}.
  \end{align}
 Let $F$ denote the equation that is  
  $q^{-1}$ times \eqref{eq:2c1} 
  plus
    $q^{-1}$ times \eqref{eq:2c2} 
    minus  \eqref{eq:2c3} 
        plus \eqref{eq:2c4} 
         minus \eqref{eq:2c5} 
           plus \eqref{eq:2c6}. The left-hand side of $F$ is equal to zero, by our preliminary comment. In the right-hand side of $F$,  eliminate $\lbrack W_0, B^+(t)\rbrack_q$ using   \eqref{eq:L1}, and
           $\lbrack B^-(s), W_1\rbrack_q$ using \eqref{eq:L2}, and $\lbrack W_0, W_1 \rbrack_q$ 
         using $\lbrack W_0, W_1 \rbrack_q=-q B_\delta$. Simplify the result to obtain \eqref{eq:cc2}.
           The equation \eqref{eq:cc3}  is obtained \eqref{eq:cc1} by applying the antiautomorphism $\tau$.
   \end{proof}
  
  \begin{proposition}
  \label{prop:wangGF}  For the algebra $O_q$,
  \begin{align*}
  0 &= 
  B(s)B^+(t) (1-q^{-2}st)(t-q^2s)
  -
  B^+(t)B(s)(1-q^2st)(t-q^{-2}s)
  \\
  &\quad +
  B(s)B^+(q^{-2}s)(q^2
-q^{-2})(1-q^{-2}st)s
+
B^-(q^{-2}s)B(s)(q^2 -q^{-2})(t-q^{-2}s)s,
\\
 0 &= 
  B^-(t)B(s)(1-q^{-2}st)(t-q^2s)
  -
  B(s)B^-(t)(1-q^2st)(t-q^{-2}s)
  \\
  &\quad +
  B^-(q^{-2}s)B(s)(q^2
-q^{-2})(1-q^{-2}st)s
+
B(s)B^+(q^{-2}s)(q^2 -q^{-2})(t-q^{-2}s)s.
\end{align*} 
\end{proposition}
\begin{proof} We will use Proposition \ref{prop:wang2}. Observe that
\begin{align}
\label{eq:m1}
\sum_{m=0}^\infty \sum_{\ell=0}^\infty \lbrack B_{(m+1)\delta}, B_{\ell \delta+\alpha_1} \rbrack s^m t^\ell &= s^{-1} \bigl\lbrack B(s), B^+(t)\bigr\rbrack,
\\
\label{eq:m2}
\sum_{m=0}^\infty \sum_{\ell=0}^\infty \lbrack B_{\ell \delta+\alpha_1}, B_{(m-1)\delta} \rbrack s^m t^\ell &= s \bigl\lbrack B^+(t), B(s)\bigr\rbrack,
\\
\label{eq:m3}
\sum_{m=0}^\infty \sum_{\ell=0}^\infty \lbrack B_{m\delta}, B_{(\ell+1) \delta+\alpha_1} \rbrack_{q^2} s^m t^\ell &= t^{-1}\bigl\lbrack B(s), B^+(t)-W_1\bigr\rbrack_{q^2},
\\
\label{eq:m4}
\sum_{m=0}^\infty \sum_{\ell=0}^\infty \lbrack B_{(\ell-1)\delta+\alpha_1}, B_{m \delta} \rbrack_{q^2} s^m t^\ell &=  \bigl\lbrack W_0+tB^+(t), B(s) \bigr\rbrack_{q^2}.
\end{align}
\noindent Let $E$ denote the equation that is \eqref{eq:m1} minus \eqref{eq:m2} minus \eqref{eq:m3} plus \eqref{eq:m4}. 
The left-hand side of $E$ is equal to zero, by Proposition \ref{prop:wang2}. After some algebraic manipulation, $E$ becomes
\begin{align}
\label{eq:m5}
\begin{split}
0 = (1-q^{-2}st)(t-q^2 s)B(s)B^+(t)&-(1-q^2st)(t-q^{-2}s)B^+(t)B(s)\\
&+ s \lbrack B(s), W_1 \rbrack_{q^2}
+st \lbrack W_0, B(s)\rbrack_{q^2}.
\end{split}
\end{align}
Setting $t=q^{-2}s$ in \eqref{eq:m5}, we obtain
\begin{align}
\label{eq:m6}
0 = (1-q^{-4}s^2)(q^{-2}-q^2) s B(s)B^+(q^{-2}s)
+ s \lbrack B(s), W_1 \rbrack_{q^2}
+q^{-2}s^2 \lbrack W_0, B(s)\rbrack_{q^2}.
\end{align}
Applying the antiautomorphism $\tau$ to the right-hand side of \eqref{eq:m5}, \eqref{eq:m6} yields
\begin{align}
\label{eq:m7}
\begin{split}
0 = (1-q^{-2}st)(t-q^2 s)B^-(t) B(s)&-(1-q^2st)(t-q^{-2}s)B(s)B^-(t)\\
&+ s \lbrack W_0, B(s)\rbrack_{q^2}
+st \lbrack  B(s), W_1\rbrack_{q^2}
\end{split}
\end{align}
and
\begin{align}
\label{eq:m8}
0 = (1-q^{-4}s^2)(q^{-2}-q^2) s B^-(q^{-2}s) B(s)
+ s \lbrack W_0, B(s) \rbrack_{q^2} 
+q^{-2}s^2 \lbrack B(s), W_1\rbrack_{q^2}.
\end{align}
The equation \eqref{eq:m5} minus $(1-q^{-2}st)(1-q^{-4}s^2)^{-1}$ times \eqref{eq:m6} minus $(t-q^{-2} s)(1-q^{-4} s^2)^{-1}$ times \eqref{eq:m8}
gives the first equation in the proposition statement.
The equation \eqref{eq:m7} minus $(t-q^{-2} s)(1-q^{-4} s^2)^{-1}$ times \eqref{eq:m6} minus $(1-q^{-2}st)(1-q^{-4}s^2)^{-1}$ times \eqref{eq:m8} 
gives the second equation in the proposition statement.
   \end{proof}

  \begin{corollary} \label{prop:pbwRel} 
  For the algebra $O_q$,
  \begin{align} 
  \label{eq:L7}
  \lbrack W_0, B^-( t) \rbrack_q &= 
  (q-q^{-1})\bigl( B^-(t )\bigr)^2
  + q B(t) +q-q^{-1},
  \\  \label{eq:L8}
  \lbrack W_0, B( t) \rbrack_{q^2} &= 
  (q^2-q^{-2}) B^-(q^{-2}t )B(t)
  -q^{-2} (q^2-q^{-2}) t B(t) B^+(q^{-2}t),
  \\ \label{eq:L9}
  \lbrack B^+( t), W_1\rbrack_q &= 
  (q-q^{-1})\bigl( B^+(t )\bigr)^2
  + q B(t) +q-q^{-1},
  \\ \label{eq:L10}
  \lbrack B( t), W_1 \rbrack_{q^2} &= 
  (q^2-q^{-2}) B(t) B^+(q^{-2}t )
  -q^{-2}(q^2-q^{-2}) t B^-(q^{-2}t) B(t).
  \end{align}
  \end{corollary}
  \begin{proof} Set $s=0$  in the first and third equation of Proposition \ref{prop:GFwang}. Set $t=0$ in Proposition  \ref{prop:wangGF}. Evaluate the results using \eqref{eq:zv}.
  \end{proof}

\noindent We mention two special cases of \eqref{eq:cc2} for later use.
\begin{corollary} \label{lem:laterUse}
For the algebra $O_q$,
\begin{align*}
0 &= 
  q^{-1}t(qt^{-1}-q^{-1} t) B^-(qt) B^+(q^{-1}t) +
  q^{-1}t(qt- q^{-1}t^{-1}) B^+(q^{-1}t) B^-(q t) \\
 &\quad + t(q-q^{-1})\bigl( B^-(qt) \bigr)^2 
  + q^{-2} t(q-q^{-1}) \bigl( B^+(q^{-1}t) \bigr)^2  
\\
 & \quad +\frac{t-t^{-1}}{q-q^{-1}} B(q^{-1} t) - \frac{t-t^{-1}}{q-q^{-1}}B(qt),
\\
0 &= 
  q^{-1}t(qt^{-1}-q^{-1} t) B^-(q^{-1}t) B^+(qt) +
  q^{-1}t(qt- q^{-1}t^{-1}) B^+(qt) B^-(q^{-1} t) \\
 &\quad + t(q-q^{-1})\bigl( B^+(qt) \bigr)^2 
  + q^{-2} t(q-q^{-1}) \bigl( B^-(q^{-1}t) \bigr)^2  
\\
 & \quad +\frac{t-t^{-1}}{q-q^{-1}} B(q^{-1} t) - \frac{t-t^{-1}}{q-q^{-1}}B(qt).
\end{align*}

\end{corollary}
\begin{proof} To obtain the first (resp. second) displayed equation,
set $s=qt'$ and $t=q^{-1}t'$ (resp. $s=q^{-1}t'$ and $t=qt'$) in \eqref{eq:cc2}. 
\end{proof}

 \section{The algebra $\mathcal O_q$}
   \noindent In the previous section we discussed the $q$-Onsager algebra $O_q$. In this section we discuss its alternating central extension $\mathcal O_q$.
  
\begin{definition}\rm
\label{def:Aq}
(See 
\cite{BK05}, \cite[Definition~3.1]{basnc}.)
Define the algebra $\mathcal O_q$
by generators
\begin{align}
\label{eq:4gens}
\lbrace \mathcal W_{-k}\rbrace_{k\in \mathbb N}, \qquad  \lbrace \mathcal  W_{k+1}\rbrace_{k\in \mathbb N},\qquad  
 \lbrace \mathcal G_{k+1}\rbrace_{k\in \mathbb N},
\qquad
\lbrace \mathcal {\tilde G}_{k+1}\rbrace_{k\in \mathbb N}
\end{align}
 and the following relations. For $k, \ell \in \mathbb N$,
\begin{align}
&
 \lbrack \mathcal W_0, \mathcal W_{k+1}\rbrack= 
\lbrack \mathcal W_{-k}, \mathcal W_{1}\rbrack=
({\mathcal{\tilde G}}_{k+1} - \mathcal G_{k+1})/(q+q^{-1}),
\label{eq:3p1}
\\
&
\lbrack \mathcal W_0, \mathcal G_{k+1}\rbrack_q= 
\lbrack {\mathcal{\tilde G}}_{k+1}, \mathcal W_{0}\rbrack_q= 
\rho  \mathcal W_{-k-1}-\rho 
 \mathcal W_{k+1},
\label{eq:3p2}
\\
&
\lbrack \mathcal G_{k+1}, \mathcal W_{1}\rbrack_q= 
\lbrack \mathcal W_{1}, {\mathcal {\tilde G}}_{k+1}\rbrack_q= 
\rho  \mathcal W_{k+2}-\rho 
 \mathcal W_{-k},
\label{eq:3p3}
\\
&
\lbrack \mathcal W_{-k}, \mathcal W_{-\ell}\rbrack=0,  \qquad 
\lbrack \mathcal W_{k+1}, \mathcal W_{\ell+1}\rbrack= 0,
\label{eq:3p4}
\\
&
\lbrack \mathcal W_{-k}, \mathcal W_{\ell+1}\rbrack+
\lbrack \mathcal W_{k+1}, \mathcal W_{-\ell}\rbrack= 0,
\label{eq:3p5}
\\
&
\lbrack \mathcal W_{-k}, \mathcal G_{\ell+1}\rbrack+
\lbrack \mathcal G_{k+1}, \mathcal W_{-\ell}\rbrack= 0,
\label{eq:3p6}
\\
&
\lbrack \mathcal W_{-k}, {\mathcal {\tilde G}}_{\ell+1}\rbrack+
\lbrack {\mathcal {\tilde G}}_{k+1}, \mathcal W_{-\ell}\rbrack= 0,
\label{eq:3p7}
\\
&
\lbrack \mathcal W_{k+1}, \mathcal G_{\ell+1}\rbrack+
\lbrack \mathcal  G_{k+1}, \mathcal W_{\ell+1}\rbrack= 0,
\label{eq:3p8}
\\
&
\lbrack \mathcal W_{k+1}, {\mathcal {\tilde G}}_{\ell+1}\rbrack+
\lbrack {\mathcal {\tilde G}}_{k+1}, \mathcal W_{\ell+1}\rbrack= 0,
\label{eq:3p9}
\\
&
\lbrack \mathcal G_{k+1}, \mathcal G_{\ell+1}\rbrack=0,
\qquad 
\lbrack {\mathcal {\tilde G}}_{k+1}, {\mathcal {\tilde G}}_{\ell+1}\rbrack= 0,
\label{eq:3p10}
\\
&
\lbrack {\mathcal {\tilde G}}_{k+1}, \mathcal G_{\ell+1}\rbrack+
\lbrack \mathcal G_{k+1}, {\mathcal {\tilde G}}_{\ell+1}\rbrack= 0.
\label{eq:3p11}
\end{align}
In the above equations $\rho = -(q^2-q^{-2})^2$. The generators 
\eqref{eq:4gens} are called {\it alternating}. We call $\mathcal O_q$ the {\it alternating
central extension of $O_q$}. 
\noindent For notational convenience define
\begin{align}
{\mathcal G}_0 = -(q-q^{-1})\lbrack 2 \rbrack^2_q, \qquad \qquad 
{\mathcal {\tilde G}}_0 = -(q-q^{-1}) \lbrack 2 \rbrack^2_q.
\label{eq:GG0}
\end{align}
\end{definition}

\begin{note}\rm In earlier papers\cite{basBel},   \cite{z2z2z2}, \cite{conj},
\cite{pbwqO} the algebra $\mathcal O_q$ is denoted by $\mathcal A_q$.
\end{note}
 
\begin{proposition} \label{lem:pbw} {\rm (See \cite[Theorem~6.1]{pbwqO}.)} A PBW basis for $\mathcal O_q$ is obtained by its alternating generators in any linear order $<$ such that
\begin{align}
\mathcal G_{i+1} < \mathcal W_{-j} < \mathcal W_{k+1} < \mathcal {\tilde G}_{\ell+1}\qquad \qquad i,j,k, \ell \in \mathbb N.
\label{eq:order}
\end{align}
\end{proposition}

\noindent Next we describe some symmetries of $\mathcal O_q$.
\begin{lemma}
\label{lem:autAc} {\rm (See \cite[Remark~1]{basBel}.)} There exists an automorphism $\sigma$ of $\mathcal O_q$ that sends
\begin{align*}
\mathcal W_{-k} \mapsto \mathcal W_{k+1}, \qquad
\mathcal W_{k+1} \mapsto \mathcal W_{-k}, \qquad
\mathcal G_{k+1} \mapsto \mathcal {\tilde G}_{k+1}, \qquad
\mathcal {\tilde G}_{k+1} \mapsto \mathcal G_{k+1}
\end{align*}
 for $k \in \mathbb N$. Moreover $\sigma^2 = {\rm id}$.
\end{lemma}

\begin{lemma}\label{lem:antiautAc} {\rm (See \cite[Lemma~3.7]{z2z2z2}.)} There exists an antiautomorphism $\dagger$ of $\mathcal O_q$ that sends
\begin{align*}
\mathcal W_{-k} \mapsto \mathcal W_{-k}, \qquad
\mathcal W_{k+1} \mapsto \mathcal W_{k+1}, \qquad
\mathcal G_{k+1} \mapsto \mathcal {\tilde G}_{k+1}, \qquad
\mathcal {\tilde G}_{k+1} \mapsto \mathcal G_{k+1}
\end{align*}
for $k \in \mathbb N$. Moreover $\dagger^2={\rm id}$.
\end{lemma}

\begin{lemma} \label{lem:sdcomAc}  {\rm (See \cite[Lemma~4.6]{compQons}.)}
The maps $\sigma$, $\dagger $ commute.
\end{lemma}

\begin{definition}\label{def:tauAc}\rm Let $\tau$ denote the composition of the automorphism $\sigma$ from Lemma \ref{lem:autAc} and the antiautomorphism $\dagger$ from Lemma \ref{lem:antiautAc}. 
 Note that $\tau$ is an antiautomorphism of $\mathcal O_q$ that sends
\begin{align*}
\mathcal W_{-k} \mapsto \mathcal W_{k+1}, \qquad
\mathcal W_{k+1} \mapsto \mathcal W_{-k}, \qquad
\mathcal G_{k+1} \mapsto \mathcal G_{k+1}, \qquad
\mathcal {\tilde G}_{k+1} \mapsto \mathcal {\tilde G}_{k+1}
\end{align*}
\noindent for $k \in \mathbb N$. We have $\tau^2={\rm id}$.
\end{definition}

\noindent Next we discuss how $\mathcal O_q$ is related to $O_q$.

\begin{lemma}
\label{lem:iota}  {\rm (See \cite[Theorem~10.3]{pbwqO}.)} 
There exists an algebra homomorphism $\imath: O_q \to \mathcal O_q$ that sends $W_0 \mapsto \mathcal W_0$ and
$W_1 \mapsto \mathcal W_1$. Moreover, $\imath$ is injective.
\end{lemma}

\begin{lemma} 
\label{lem:diag} {\rm (See \cite[Lemma~4.11]{compQons}.)}
The following diagrams commute:

\begin{equation*}
{\begin{CD}
O_q @>\imath  >> \mathcal O_q
              \\
         @V \sigma VV                   @VV \sigma V \\
         O_q @>>\imath >
                                 \mathcal O_q
                        \end{CD}}  	
         \qquad \qquad                
    {\begin{CD}
O_q @>\imath  >> \mathcal O_q
              \\
         @V \dagger VV                   @VV \dagger V \\
         O_q @>>\imath >
                                 \mathcal O_q
                        \end{CD}}  	     \qquad \qquad                
   {\begin{CD}
O_q @>\imath  >> \mathcal O_q
              \\
         @V \tau VV                   @VV \tau V \\
         O_q @>>\imath >
                                 \mathcal O_q
                        \end{CD}}  	                                       		    
\end{equation*}
\end{lemma}

\section{Generating functions for $\mathcal O_q$}

\noindent In Definition \ref{def:Aq} the algebra $\mathcal O_q$ is defined by generators and relations.
In this section we describe the defining relations in terms of generating functions.

\begin{definition}
\label{def:gf4}
\rm
We define some generating functions in the indeterminate $t$:
\begin{align*}
&\mathcal W^-(t) = \sum_{n \in \mathbb N} \mathcal W_{-n} t^n,
\qquad \qquad  
\mathcal W^+(t) = \sum_{n \in \mathbb N} \mathcal W_{n+1} t^n,
\\
&\mathcal G(t) = \sum_{n \in \mathbb N} \mathcal G_n t^n,
\qquad \qquad \qquad 
\mathcal {\tilde G}(t) = \sum_{n \in \mathbb N} \mathcal {\tilde G}_n t^n.
\end{align*}
\end{definition}
\noindent Observe that
\begin{align*}
\mathcal W^-(0) = \mathcal W_0, \qquad
\mathcal W^+(0) = \mathcal W_1, \qquad
\mathcal G(0) = -(q-q^{-1}) \lbrack 2 \rbrack^2_q, \qquad
\mathcal {\tilde G}(0) = -(q-q^{-1}) \lbrack 2 \rbrack^2_q.
\end{align*}

\noindent  We now give the relations \eqref{eq:3p1}--\eqref{eq:3p11}  in terms of generating functions. 

\begin{lemma} \label{lem:ad} {\rm (See \cite[Lemma~3.6]{pbwqO}.)}
 For the algebra $\mathcal O_q$ we have
\begin{align}
& \label{eq:3pp1}
\lbrack \mathcal W_0, \mathcal W^+(t) \rbrack = \lbrack \mathcal W^-(t), \mathcal W_1 \rbrack = t^{-1}(\mathcal {\tilde G}(t)-\mathcal G(t))/(q+q^{-1}),
\\
& \label{eq:3pp2}
\lbrack \mathcal W_0, \mathcal G(t) \rbrack_q = \lbrack \mathcal {\tilde G}(t), \mathcal W_0 \rbrack_q = \rho \mathcal W^-(t)-\rho t \mathcal W^+(t),
\\
&\label{eq:3pp3}
\lbrack \mathcal G(t), \mathcal W_1 \rbrack_q = \lbrack \mathcal W_1, \mathcal {\tilde G}(t) \rbrack_q = \rho \mathcal W^+(t) -\rho t \mathcal W^-(t),
\\
&\label{eq:3pp4}
\lbrack  \mathcal W^-(s), \mathcal W^-(t) \rbrack = 0, 
\qquad 
\lbrack \mathcal W^+(s),  \mathcal W^+(t) \rbrack = 0,
\\ \label{eq:3pp5}
&\lbrack  \mathcal W^-(s), \mathcal W^+(t) \rbrack 
+
\lbrack \mathcal W^+(s), \mathcal W^-(t) \rbrack = 0,
\\ \label{eq:3pp6}
&s \lbrack \mathcal W^-(s), \mathcal G(t) \rbrack 
+
t \lbrack  \mathcal G(s),  \mathcal W^-(t) \rbrack = 0,
\\ \label{eq:3pp7}
&s \lbrack  \mathcal W^-(s), \mathcal {\tilde G}(t) \rbrack 
+
t \lbrack  \mathcal {\tilde G}(s), \mathcal W^-(t) \rbrack = 0,
\\ \label{eq:3pp8}
&s \lbrack   \mathcal W^+(s),  \mathcal G(t) \rbrack
+
t \lbrack   \mathcal G(s), \mathcal W^+(t) \rbrack = 0,
\\ \label{eq:3pp9}
&s \lbrack   \mathcal W^+(s), \mathcal {\tilde G}(t) \rbrack
+
t \lbrack \mathcal {\tilde G}(s), \mathcal W^+(t) \rbrack = 0,
\\ \label{eq:3pp10}
&\lbrack   \mathcal G(s), \mathcal G(t) \rbrack = 0, 
\qquad 
\lbrack  \mathcal {\tilde G}(s),  \mathcal {\tilde G}(t) \rbrack = 0,
\\ \label{eq:3pp11}
&\lbrack  \mathcal {\tilde G}(s), \mathcal G(t) \rbrack +
\lbrack   \mathcal G(s), \mathcal {\tilde G}(t) \rbrack = 0.
\end{align}
\end{lemma}
\begin{proof} Use \eqref{eq:3p1}--\eqref{eq:3p11} and Definition
\ref{def:gf4}.
\end{proof}

\section{The algebra isomorphism $\phi: O_q \otimes \mathbb F \lbrack z_1, z_2,\ldots \rbrack \to \mathcal O_q$}
 
\noindent 
 Let $\mathcal Z$ denote the center of $\mathcal O_q$. In this section we describe $\mathcal Z$ from various points of view. Using this
 description we will obtain an algebra isomorphism $\phi: O_q \otimes \mathbb F \lbrack z_1, z_2,\ldots \rbrack \to \mathcal O_q$.
 \medskip
 
 \noindent For notational convenience define
\begin{align}
S=\frac{q+q^{-1}}{q^{-1} t+ q t^{-1}}, \qquad \qquad 
T=\frac{q+q^{-1}}{qt + q^{-1} t^{-1}}.
\label{eq:ST}
\end{align}
\noindent We view $S$ and $T$ as power series
\begin{align*}
S = (q+q^{-1}) \sum_{\ell \in \mathbb N} (-1)^\ell q^{-2\ell-1} t^{2\ell+1} \qquad \qquad T = (q+q^{-1}) \sum_{\ell \in \mathbb N} (-1)^\ell q^{2\ell+1} t^{2\ell+1}.
\end{align*}

\begin{definition} \label{lem:zV1} \rm (See \cite[Definition~8.4]{pbwqO}.)
 For the algebra $\mathcal O_q$ define
\begin{align}
\begin{split} \label{eq:Zlong}
\Psi(t) &= 
t^{-1} ST\mathcal W^-(S) \mathcal W^+(T) +
t ST\mathcal W^+(S) \mathcal W^-(T) -
q^2 ST \mathcal W^-(S) \mathcal W^-(T)
\\
& \quad \qquad -
q^{-2}ST \mathcal W^+(S) \mathcal W^+(T) +
(q^2-q^{-2})^{-2} \mathcal G(S) \mathcal {\tilde G}(T).
\end{split}
\end{align}
\end{definition}

\begin{note}\rm In  \cite[Definition~8.4]{pbwqO} the generating function $\Psi(t)$ is called $\mathcal Z(t)$. \end{note}
\noindent The following normalization is sometimes convenient.
\begin{definition} \label{def:ZV} \rm
Define
\begin{align}
\mathcal Z^\vee(t) = \lbrack 2 \rbrack^{-2}_q \Psi(t).
\label{eq:norm}
\end{align}
\end{definition}

 \begin{definition}\label{def:Zn} For $n \in \mathbb N$ define $\mathcal Z^\vee_n \in \mathcal O_q$ such that
 \begin{align*}
 \mathcal Z^\vee(t) = \sum_{n\in \mathbb N} \mathcal Z^\vee_n  t^n.
 \end{align*}
 \end{definition}

\noindent By \cite[Lemma~8.18]{pbwqO} we have $\mathcal Z^\vee_0 = 1$.

\begin{lemma} \label{lem:Zfix} {\rm (See \cite[Lemma~8.10 and Proposition~8.12]{pbwqO}.)}
For $n\geq 1$ we have $\mathcal Z^\vee_n\in \mathcal Z$. Moreover $\mathcal Z^\vee_n$  fixed by $\sigma$ and $\dagger$ and $\tau$.
\end{lemma}

\begin{definition}\rm Let $\langle \mathcal W_0, \mathcal W_1 \rangle$ denote the subalgebra of $\mathcal O_q$ generated by $\mathcal W_0$, $\mathcal W_1$.
\end{definition}

\begin{proposition} \label{lem:sum}
{\rm (See \cite[Theorems~10.2--10.4]{pbwqO}.)} For the algebra $\mathcal O_q$ the following {\rm (i)--(iii)} hold:
\begin{enumerate}
 \item[\rm (i)] 
there exists an algebra isomorphism
$O_q \to \langle \mathcal W_0, \mathcal W_1\rangle$ 
that sends $W_0\mapsto \mathcal W_0$ and
$W_1\mapsto \mathcal W_1$;
\item[\rm (ii)] 
there exists an algebra isomorphism
 $\mathbb F\lbrack z_1, z_2,\ldots \rbrack \to \mathcal Z$
 that sends $z_n \mapsto \mathcal Z^\vee_n$ for $n \geq 1$;
\item[\rm (iii)] 
the multiplication map 
\begin{align*}
\langle \mathcal W_0, \mathcal W_1\rangle 
\otimes
\mathcal Z
 & \to   \mathcal O_q
\\
 w \otimes z  &\mapsto  wz
 \end{align*}
 is an isomorphism of algebras.
 \end{enumerate}
\end{proposition}

\noindent Note that the isomorphism in Proposition \ref{lem:sum}(i) is induced by the map $\imath$ from Lemma \ref{lem:iota}.


 \begin{proposition}\label{prop:three} {\rm (See \cite[Theorem~9.14]{pbwqO}.)}
There exists an algebra isomorphism $\phi: O_q \otimes \mathbb F \lbrack z_1, z_2,\ldots \rbrack \to \mathcal O_q$ that sends
\begin{align*}
W_0 \otimes 1 \mapsto \mathcal W_0, \qquad \quad
W_1 \otimes 1 \mapsto \mathcal W_1, \qquad \quad 
1 \otimes z_n \mapsto \mathcal Z^\vee_n, \qquad n\geq 1.
\end{align*}
\end{proposition}

\begin{proposition} 
\label{lem:Zdiag} 
The following diagrams commute:

\begin{align*}
&{\begin{CD}
O_q \otimes \mathbb F \lbrack z_1, z_2, \ldots \rbrack @>\phi  >> \mathcal O_q
              \\
         @V \sigma\otimes {\rm id}  VV                   @VV \sigma V \\
     O_q \otimes \mathbb F \lbrack z_1, z_2, \ldots \rbrack   @>>\phi >
                                 \mathcal O_q
                        \end{CD}}  	
         \qquad       \qquad     
    {\begin{CD}
O_q \otimes \mathbb F \lbrack z_1, z_2, \ldots \rbrack@>\phi >> \mathcal O_q
              \\
         @V \dagger \otimes {\rm id} VV                   @VV \dagger V \\
     O_q \otimes \mathbb F \lbrack z_1, z_2, \ldots \rbrack     @>>\phi>
                                 \mathcal O_q
                        \end{CD}}  	   
                        \\          
&  {\begin{CD}
O_q \otimes \mathbb F \lbrack z_1, z_2, \ldots \rbrack @>\phi >> \mathcal O_q
              \\
         @V\tau \otimes {\rm id}VV                   @VV \tau V \\
      O_q \otimes \mathbb F \lbrack z_1, z_2, \ldots \rbrack  @>>\phi >
                                 \mathcal O_q
                        \end{CD}}  	                              		    
\end{align*}
\end{proposition}
\begin{proof} Chase the generators $W_0 \otimes 1$, $W_1 \otimes 1$, $\lbrace 1\otimes z_n \rbrace_{n=1}^\infty$ around each diagram using
Lemmas \ref{lem:aut}, \ref{lem:antiaut}, \ref{lem:autAc}, \ref{lem:antiautAc}
and 
Definitions \ref{def:tauA}, \ref{def:tauAc}
along with
Lemma \ref{lem:Zfix} and Proposition \ref{prop:three}.
\end{proof}

\noindent We emphasize a few points.
\begin{corollary}\label{cor:sum}
 The following {\rm (i)--(iii)} hold:
 \begin{enumerate}
 \item[\rm (i)] the algebra $\mathcal O_q$ is generated by $\mathcal W_0$, $\mathcal W_1$, $\mathcal Z$;
 \item[\rm (ii)] the elements $\lbrace \mathcal Z^\vee_n \rbrace_{n=1}^\infty$ are algebraically independent and generate $\mathcal Z$;
 \item[\rm (iii)] everything in $\mathcal Z$ is fixed by $\sigma$ and $\dagger$ and $\tau$.
 \end{enumerate}
 \end{corollary}
 \begin{proof} (i) By Proposition  \ref{lem:sum}(iii). \\
 \noindent (ii) By Proposition  \ref{lem:sum}(ii). \\
 \noindent (iii) By (ii) above and Lemma  \ref{lem:Zfix}.
 \end{proof}

\section{Comparing the generating functions for $O_q$ and $\mathcal O_q$}
\noindent In this section we investigate how the generating functions $B^\pm(t)$, $B(t)$ for $O_q$
are related to the generating functions $\mathcal W^\pm(t)$, $\mathcal G(t)$, $\mathcal {\tilde G}(t)$ for $\mathcal O_q$.
\medskip

\noindent Throughout this section, we identify $O_q$ with $\langle \mathcal W_0, \mathcal W_1 \rangle $ via the map $\imath$ from 
Lemma \ref{lem:iota}.

\begin{lemma} \label{lem:G1} {\rm (See \cite[Lemma~11.3]{compQons}.)}
The element $\mathcal {\tilde G}_1+ q B_{\delta}$ is central in $\mathcal O_q$.
\end{lemma}

\begin{lemma}\label{lem:Ogen} For the algebra $\mathcal O_q$,
\begin{align}
\lbrack B_\delta, \mathcal {\tilde G}_n\rbrack = 0, \qquad \qquad n \in \mathbb N.
\end{align}
\end{lemma}
\begin{proof} By Lemma \ref{lem:G1} and since $\lbrack \mathcal {\tilde G}_1, \mathcal {\tilde G}_n\rbrack = 0$ for $n \in \mathbb N$.
\end{proof}

\begin{lemma} \label{lem:GBd} For the algebra $\mathcal O_q$,
\begin{align}
\lbrack B_\delta, \mathcal {\tilde G}(t) \rbrack = 0.
\end{align}
\end{lemma}
\begin{proof} By Lemma \ref{lem:Ogen}.
\end{proof}

\begin{lemma}\label{lem:Wind} {\rm (See \cite[Lemma~11.5]{compQons}.)}
For $n\geq 1$ the following hold in $\mathcal O_q$:
\begin{align}
\label{eq:Wind1}
\mathcal W_{-n} &= \mathcal W_n -\frac{(q-q^{-1}) \mathcal W_0 \mathcal{\tilde G}_n}{(q^2-q^{-2})^2} + \frac{q^2 \lbrack B_\delta, \mathcal W_{1-n}\rbrack}{(q^2-q^{-2})^2},
\\
\mathcal W_{n+1} &=\mathcal W_{1-n}-\frac{(q-q^{-1}) \mathcal W_1\mathcal {\tilde G}_n}{(q^2-q^{-2})^2} -\frac{\lbrack B_\delta, \mathcal W_n\rbrack}{(q^2-q^{-2})^2}.
\label{eq:Wind2}
\end{align}
\end{lemma}

\begin{lemma}\label{lem:Wind2} 
For the algebra $\mathcal O_q$,
\begin{align}
\label{eq:Wind1c}
\mathcal W^-(t) &=t \mathcal W^+(t) -\frac{(q-q^{-1}) \mathcal W_0 \mathcal{\tilde G}(t)}{(q^2-q^{-2})^2} + \frac{q^2 t \lbrack B_\delta, \mathcal W^-(t)\rbrack}{(q^2-q^{-2})^2},
\\
\mathcal W^+(t) &=t \mathcal W^-(t)-\frac{(q-q^{-1}) \mathcal W_1\mathcal {\tilde G}(t)}{(q^2-q^{-2})^2} -\frac{t \lbrack B_\delta, \mathcal W^+(t)\rbrack}{(q^2-q^{-2})^2}.
\label{eq:Wind2c}
\end{align}
\end{lemma}
\begin{proof} The equation \eqref{eq:Wind1c} (resp.  \eqref{eq:Wind2c}) expresses the recurrence \eqref{eq:Wind1} (resp.  \eqref{eq:Wind2}) in terms of generating functions.
\end{proof}

\noindent Recall  $S$, $T$ from \eqref{eq:ST}. We will be discussing  $\bigl(\mathcal {\tilde G}(S)\bigr)^{-1} $ and $\bigl(\mathcal {\tilde G}(T)\bigr)^{-1} $.
These inverses exist by 
 \cite[Lemmas~4.1, 4.6]{conj}.

\begin{proposition} \label{BsolveV} For the algebra $\mathcal O_q$,
\begin{align}
B^-(t) &= (q^2-q^{-2}) S \Bigl( q^{-1}\mathcal W^+(S)-qt^{-1} \mathcal W^-(S)\Bigr)  \bigl(\mathcal {\tilde G}(S)\bigr)^{-1} ,
\label{eq:Bmsolve2}
\\
B^+(t) &= (q^2-q^{-2}) T \Bigl( q \mathcal W^-(T)-q^{-1}t^{-1} \mathcal W^+(T)\Bigr) \bigl(\mathcal {\tilde G}(T)\bigr)^{-1}.
\label{eq:Bpsolve2}
\end{align}
\end{proposition}
\begin{proof} Let $b^-(t) $ denote the expression on the right in \eqref{eq:Bmsolve2}. We show that $B^-(t)=b^-(t)$. 
Using \cite[Sections~4, 5]{conj} we examine the terms on the right in \eqref{eq:Bmsolve2}, and find that
 $b^-(t)$ has the form $b^-(t) = \sum_{n \in \mathbb N} b_n t^n$ with $b_n \in \mathcal O_q$ for $n \in \mathbb N$.
 We show that $b_n = B_{n\delta+\alpha_0}$ for $n \in \mathbb N$.
Using Lemmas  \ref{lem:GBd}, \ref{lem:Wind2}  we obtain
\begin{align}
\label{eq:BP1e}
&\frac{q \lbrack B_\delta, b^-(t) \rbrack}{(q-q^{-1})(q^2-q^{-2})}= 
(t^{-1}-t)b^-(t)-t^{-1} \mathcal W_0 -\mathcal W_1.
\end{align}
From \eqref{eq:BP1e} we obtain the recursion
\begin{align*}
&b_0=\mathcal W_0,  \qquad \qquad 
b_1 = \mathcal W_1 + 
\frac{q \lbrack B_{\delta}, \mathcal W_0\rbrack}{(q-q^{-1})(q^2-q^{-2})},
\\
&
b_n = b_{n-2}
+ 
\frac{q \lbrack B_{\delta}, b_{n-1}\rbrack}{(q-q^{-1})(q^2-q^{-2})} \qquad \qquad n\geq 2.
\end{align*}
Comparing this recursion with
\eqref{eq:line1}, 
\eqref{eq:line2}
we obtain $b_n = B_{n\delta+\alpha_0}$ for $n \in \mathbb N$. Therefore $B^-(t)=b^-(t)$, so \eqref{eq:Bmsolve2} holds. A similar argument that yields \eqref{eq:Bpsolve2} is summarized as follows.
Let $b^+(t) $ denote the generating function on the right in \eqref{eq:Bpsolve2}. Using Lemmas  \ref{lem:GBd}, \ref{lem:Wind2}  we obtain
\begin{align}
\label{eq:BP2e}
&\frac{q \lbrack b^+(t), B_\delta\rbrack}{(q-q^{-1})(q^2-q^{-2})}= 
(t^{-1}-t)b^+(t)-\mathcal W_0-t^{-1} \mathcal W_1.
\end{align}
Comparing \eqref{eq:BP2} and \eqref{eq:BP2e}, we find that the coefficients of $B^+(t)$ and $b^+(t)$ satisfy the same recurrence and initial conditions.
Therefore these coefficients coincide,  so $B^+(t)=b^+(t)$ and  \eqref{eq:Bpsolve2} holds.
\end{proof}
\noindent Next we give a variation on Proposition  \ref{BsolveV}.

\begin{proposition} \label{prop:Bsolve} For the algebra $\mathcal O_q$,
\begin{align}
B^-(t) &= (q^2-q^{-2}) \bigl(\mathcal {\tilde G}(T)\bigr)^{-1}  \Bigl( q \mathcal W^+(T)-q^{-1}t^{-1} \mathcal W^-(T)\Bigr) T,
\label{eq:Bmsolve}
\\
B^+(t) &= (q^2-q^{-2}) \bigl(\mathcal {\tilde G}(S)\bigr)^{-1} \Bigl( q^{-1} \mathcal W^-(S)-qt^{-1} \mathcal W^+(S)\Bigr)S.
\label{eq:Bpsolve}
\end{align}
\end{proposition}
\begin{proof} 
Apply the antiautomorphism $\tau$ to everything in Proposition   \ref{BsolveV}.
\end{proof}

\begin{theorem} \label{lem:zzznote}
For the algebra $\mathcal O_q$,
\begin{align}
\label{eq:long1}
\frac{q+q^{-1}}{t+t^{-1}}
\mathcal W^-\biggl( \frac{q+q^{-1}}{t+t^{-1}}\biggr) &=
\frac{q^{-1}t B^+(q^{-1}t)+ B^-(qt)    
}{(q^2-q^{-2})(t-t^{-1})}\,
\mathcal {\tilde G}\biggl( \frac{q+q^{-1}}{t+t^{-1}}\biggr),
\\
\label{eq:long2}
\frac{q+q^{-1}}{t+t^{-1}}
\mathcal W^+\biggl( \frac{q+q^{-1}}{t+t^{-1}}\biggr) &=
\frac{ B^+(q^{-1}t)+ qtB^-(qt) }{(q^2-q^{-2})(t-t^{-1})}\,
\mathcal {\tilde G}\biggl( \frac{q+q^{-1}}{t+t^{-1}}\biggr).
\end{align}
\end{theorem}
\begin{proof} To verify these equations, evaluate $B^+(q^{-1}t)$ and $B^-(qt)$ using
Proposition \ref{BsolveV} and simplify the result using \eqref{eq:ST}.
\end{proof}

\begin{theorem} \label{lem:zzznotez}
For the algebra $\mathcal O_q$,
\begin{align}
\label{eq:long3}
\frac{q+q^{-1}}{t+t^{-1}}
\mathcal W^-\biggl( \frac{q+q^{-1}}{t+t^{-1}}\biggr) &=
\mathcal {\tilde G}\biggl( \frac{q+q^{-1}}{t+t^{-1}}\biggr)
\frac{qt B^+(qt)+ B^-(q^{-1}t)    
}{(q^2-q^{-2})(t-t^{-1})},
\\
\label{eq:long4}
\frac{q+q^{-1}}{t+t^{-1}}
\mathcal W^+\biggl( \frac{q+q^{-1}}{t+t^{-1}}\biggr) &=
\mathcal {\tilde G}\biggl( \frac{q+q^{-1}}{t+t^{-1}}\biggr)
\frac{ B^+(qt)+ q^{-1}tB^-(q^{-1}t) }{(q^2-q^{-2})(t-t^{-1})}.
\end{align}
\end{theorem}
\begin{proof} To verify these equations, evaluate $B^+(qt)$ and $B^-(q^{-1}t)$ using
Proposition \ref{prop:Bsolve} and simplify the result using \eqref{eq:ST}. Alternatively, apply the antiautomorphism $\tau$ everywhere in Theorem  \ref{lem:zzznote}.
\end{proof}

\noindent In the next two results we give some consequences of Theorems  \ref{lem:zzznote},  \ref{lem:zzznotez}.

\begin{proposition}\label{prop:extra1} For the algebra $\mathcal O_q$,
\begin{align*}
&q \mathcal {\tilde G}\biggl( \frac{q+q^{-1}}{t+t^{-1}}\biggr) \mathcal W_0 - q^{-1} \mathcal W_0 \mathcal {\tilde G}\biggl( \frac{q+q^{-1}}{t+t^{-1}}\biggr) 
\\
& =
\frac{q(q-q^{-1})(qt-q^{-1}t^{-1})B^-(qt)-t(q-q^{-1})(q^{-1}t-qt^{-1}) B^+(q^{-1}t)}{t-t^{-1}} \,
\mathcal {\tilde G}\biggl( \frac{q+q^{-1}}{t+t^{-1}}\biggr)
\\
& =
\mathcal {\tilde G}\biggl( \frac{q+q^{-1}}{t+t^{-1}}\biggr)\,
\frac{
q^{-1}(q-q^{-1})(q^{-1}t-qt^{-1}) B^-(q^{-1}t)
-t(q-q^{-1})(qt-q^{-1}t^{-1})B^+(qt)}{t-t^{-1}
}
\end{align*}
and also
\begin{align*}
&q \mathcal W_1 \mathcal {\tilde G}\biggl( \frac{q+q^{-1}}{t+t^{-1}}\biggr) -q^{-1}
\mathcal {\tilde G}\biggl( \frac{q+q^{-1}}{t+t^{-1}}\biggr) \mathcal W_1
\\
& =
\frac{
q^{-1}(q-q^{-1})(q^{-1}t-qt^{-1}) B^+(q^{-1}t)
-t(q-q^{-1})(qt-q^{-1}t^{-1})B^-(qt)
}{t-t^{-1}} \,
\mathcal {\tilde G}\biggl( \frac{q+q^{-1}}{t+t^{-1}}\biggr)
\\
& =
\mathcal {\tilde G}\biggl( \frac{q+q^{-1}}{t+t^{-1}}\biggr) \,
\frac{q(q-q^{-1})(qt-q^{-1}t^{-1})B^+(qt)-t (q-q^{-1})(q^{-1}t-qt^{-1}) B^-(q^{-1}t)}{t-t^{-1}}.
\end{align*}
\end{proposition}
\begin{proof}In the equation on the right in \eqref{eq:3pp2}, replace $t$ by $(q+q^{-1})(t+t^{-1})^{-1}$ to obtain
\begin{align*}
q \mathcal {\tilde G}\biggl( \frac{q+q^{-1}}{t+t^{-1}}\biggr) \mathcal W_0 -
q^{-1} \mathcal W_0 \mathcal {\tilde G}\biggl( \frac{q+q^{-1}}{t+t^{-1}}\biggr) =
\rho \mathcal W^-\biggl( \frac{q+q^{-1}}{t+t^{-1}}\biggr)
-\rho \frac{q+q^{-1}}{t+t^{-1}}\mathcal W^+\biggl( \frac{q+q^{-1}}{t+t^{-1}}\biggr).
\end{align*}
In the above equation, evaluate the right-hand side using Theorem \ref{lem:zzznote} (resp. Theorem \ref{lem:zzznotez})
to obtain the the first (resp. second) equation
in the proposition statement. Using $\tau$ we obtain the third and fourth equation in the proposition statement.
\end{proof}

\begin{proposition} \label{lem:ex2}
For the algebra $\mathcal O_q$, the generating function
\begin{align}
\label{eq:GF}
&\qquad \mathcal G\biggl( \frac{q+q^{-1}}{t+t^{-1}}\biggr) 
\end{align}
is equal to each of the following:
\begin{align*}
&\biggl( \frac{B(qt)}{q^{-2}-1} + 
\frac{qt(q-q^{-1})\bigl(B^-(qt)\bigr)^2 -t(q^{-1}t-qt^{-1}) B^-(qt) B^+(q^{-1}t)}{t-t^{-1}} \biggr)
\mathcal {\tilde G}\biggl( \frac{q+q^{-1}}{t+t^{-1}}\biggr),
\\
&\biggl( \frac{B(q^{-1}t)}{q^{-2}-1} -
\frac{q^{-1}t(q-q^{-1})\bigl(B^+(q^{-1}t)\bigr)^2 +t(qt-q^{-1}t^{-1}) B^+(q^{-1}t) B^-(qt)}{t-t^{-1}} \biggr)
\mathcal {\tilde G}\biggl( \frac{q+q^{-1}}{t+t^{-1}}\biggr),
\\
&\mathcal {\tilde G}\biggl( \frac{q+q^{-1}}{t+t^{-1}}\biggr) \biggl( \frac{B(qt)}{q^{-2}-1} + 
\frac{qt(q-q^{-1})\bigl(B^+(qt)\bigr)^2 -t(q^{-1}t-qt^{-1}) B^-(q^{-1}t) B^+(qt)}{t-t^{-1}} \biggr),
\\
&\mathcal {\tilde G}\biggl( \frac{q+q^{-1}}{t+t^{-1}}\biggr) \biggl( \frac{B(q^{-1}t)}{q^{-2}-1} -
\frac{q^{-1}t(q-q^{-1})\bigl(B^-(q^{-1}t)\bigr)^2 +t(qt-q^{-1}t^{-1}) B^+(qt) B^-(q^{-1}t)}{t-t^{-1}} \biggr).
\end{align*}
\end{proposition}
\begin{proof} We first show that the generating function \eqref{eq:GF} is equal to the
first of the four given expressions.
 By  \eqref{eq:3pp1},
 \begin{align}
 \lbrack \mathcal W_0, \mathcal W^+(t) \rbrack = \frac{\mathcal {\tilde G}(t)-\mathcal G(t)}{t(q+q^{-1})}. \label{eq:ttex}
 \end{align}
 In \eqref{eq:ttex}, replace $t$ by $(q+q^{-1})(t+t^{-1})^{-1}$  to obtain
 \begin{align*}
 \mathcal W_0 \mathcal W^+ \biggl( \frac{q+q^{-1}}{t+t^{-1}}\biggr)  -
 \mathcal W^+ \biggl( \frac{q+q^{-1}}{t+t^{-1}}\biggr)\mathcal W_0  = \frac{t+t^{-1}}{(q+q^{-1})^2} 
 \biggl( \mathcal {\tilde G}\biggl( \frac{q+q^{-1}}{t+t^{-1}}\biggr)-
  \mathcal G\biggl( \frac{q+q^{-1}}{t+t^{-1}}\biggr) \biggr).
\end{align*}
In the above equation, eliminate the $\mathcal W^+$ terms using \eqref{eq:long2}, and pull the resulting $\mathcal {\tilde G}$ terms to the right using
the first equation in Proposition \ref{prop:extra1}.  In the resulting equation, eliminate $\lbrack \mathcal W_0, B^+(q^{-1}t) \rbrack_q$ and 
 $\lbrack \mathcal W_0, B^-(qt) \rbrack_q$  using   \eqref{eq:L1} and   \eqref{eq:L7}, respectively. In the resulting equation, eliminate
$B(q^{-1}t)$ using the first equation in Corollary  \ref{lem:laterUse}. By the resulting equation, the generating function \eqref{eq:GF} is equal to the
first of the four given expressions. The first and second given expressions  are equal by the first equation in
Corollary \ref{lem:laterUse}. Using the antiautomorphism $\tau$ we find that the  generating function \eqref{eq:GF} is equal to the
third and fourth given expressions.
\end{proof}

\section{A factorization of $\mathcal Z^\vee(t)$}


\noindent Throughout this section we identify $O_q$ with
$\langle \mathcal W_0, \mathcal W_1\rangle $ via the map $\imath$ from Lemma
\ref{lem:iota}. Recall the generating function $\mathcal Z^\vee(t)$ from  Definition \ref{def:ZV}.
We will prove the following result.

\begin{theorem}  \label{thm:Mn} 
For the algebra $\mathcal O_q$ we have
\begin{align}
\mathcal Z^\vee(t)  = \xi \mathcal {\tilde G}(S) B(t) \mathcal {\tilde G}(T),
\label{eq:mainRes}
\end{align}
\noindent where
\begin{align}
\xi = -q (q-q^{-1}) (q^2-q^{-2})^{-4}.
\label{eq:zeta}
\end{align}
\end{theorem}
\noindent Let us consider what is needed to prove Theorem \ref{thm:Mn}.
 Evaluating the left-hand side of \eqref{eq:mainRes} using Definitions  \ref{lem:zV1}, \ref{def:ZV}
 and then rearranging terms, we find that Theorem \ref{thm:Mn} asserts the following:  zero is equal to
 \begin{align}
&\;\; t^{-1} \Bigl( \bigl(\mathcal {\tilde G}(S)\bigr)^{-1} \mathcal W^-(S) S \Bigr) \Bigl( T \mathcal W^+(T) \bigl(\mathcal {\tilde G}(T)\bigr)^{-1}\Bigr)
\label{eq:TT1}
\\
&+t \Bigl( \bigl(\mathcal {\tilde G}(S)\bigr)^{-1} \mathcal W^+(S) S \Bigr) \Bigl( T \mathcal W^-(T) \bigl(\mathcal {\tilde G}(T)\bigr)^{-1}\Bigr)
\label{eq:TT2}
\\
& -q^2 \Bigl( \bigl(\mathcal {\tilde G}(S)\bigr)^{-1} \mathcal W^-(S) S \Bigr) \Bigl( T \mathcal W^-(T) \bigl(\mathcal {\tilde G}(T)\bigr)^{-1}\Bigr)
\label{eq:TT3}
\\
& -q^{-2}\Bigl( \bigl(\mathcal {\tilde G}(S)\bigr)^{-1} \mathcal W^+(S) S \Bigr) \Bigl( T \mathcal W^+(T) \bigl(\mathcal {\tilde G}(T)\bigr)^{-1}\Bigr)
\label{eq:TT4}
\\
& + \frac{ \bigl(\mathcal {\tilde G}(S)\bigr)^{-1}\mathcal G(S) }{(q^2-q^{-2})^2}
\label{eq:TT5}
\\
&- \frac{B(t)}{(q^2-q^{-2})^2(q^{-2}-1)}.
\label{eq:TT6}
\end{align} 
\noindent We will verify the above assertion. To prepare for this, we evaluate the terms in \eqref{eq:TT1}--\eqref{eq:TT5}.

 \begin{lemma}\label{cor:zzznotez}
For the algebra $\mathcal O_q$,
\begin{align}
\label{eq:corlong1s}
\bigl(\mathcal {\tilde G}(S)\bigr)^{-1} \mathcal W^-(S) S
 &=
\frac{t B^+(t)+ B^-(q^{-2} t)}{(q^2-q^{-2})(q^{-1}t-qt^{-1})},
\\
\label{eq:corlong2s}
\bigl(\mathcal {\tilde G}(S)\bigr)^{-1} 
\mathcal W^+(S) S
 &=
\frac{B^+(t)+ q^{-2}tB^-(q^{-2} t)}{(q^2-q^{-2})(q^{-1}t-qt^{-1})}.
\end{align}
\end{lemma}
\begin{proof} In Theorem \ref{lem:zzznotez}, replace $t$ by $q^{-1}t$ and evaluate the result using \eqref{eq:ST}.
\end{proof}

 \begin{lemma}\label{cor:zzznote}
 For the algebra $\mathcal O_q$,
\begin{align}
\label{eq:corlong1}
T
\mathcal W^-(T) 
\bigl(\mathcal {\tilde G}(T)\bigr)^{-1} 
 &=
\frac{t B^+(t)+ B^-(q^2 t)}{(q^2-q^{-2})(qt-q^{-1}t^{-1})},
\\
\label{eq:corlong2}
T
\mathcal W^+(T) 
\bigl(\mathcal {\tilde G}(T)\bigr)^{-1} 
 &=
\frac{B^+(t)+ q^2tB^-(q^2 t)}{(q^2-q^{-2})(qt-q^{-1}t^{-1})}.
\end{align}
\end{lemma}
\begin{proof} In Theorem \ref{lem:zzznote}, replace $t$ by $qt$ and evaluate the result using \eqref{eq:ST}.
\end{proof}

\begin{lemma} 
\label{lem:GG}
For the algebra $\mathcal O_q$,
\begin{align*}
 \bigl(\mathcal {\tilde G}(S)\bigr)^{-1} \mathcal G(S) &= 
 \frac{B(t)}{q^{-2}-1} +
 \frac{ (q-q^{-1}) t \bigl( B^+(t)\bigr)^2 - q^{-1} t (q^{-2}t-q^2 t^{-1}) B^-(q^{-2}t) B^+(t)}{q^{-1}t-q t^{-1}}.
 \end{align*}
 \end{lemma}
 \begin{proof} In the third equation of Proposition
  \ref{lem:ex2},
 replace $t$ by $q^{-1}t$ and evaluate the result using \eqref{eq:ST}.
 \end{proof}
\noindent We can now easily prove Theorem \ref{thm:Mn}.
\medskip

\noindent {\it Proof of Theorem \ref{thm:Mn}}. It suffices to show that the sum of \eqref{eq:TT1}--\eqref{eq:TT6} is zero. This is routinely shown by evaluating the terms
using Lemmas \ref{cor:zzznotez}, \ref{cor:zzznote}, \ref{lem:GG}.  \hfill $\Box $ \\

\noindent Next, we give some consequences of Theorem \ref{thm:Mn}.

\begin{corollary} \label{cor:BZ} For the algebra $\mathcal O_q$ we have
\begin{align}
\label{eq:BZ}
B(t) = \xi^{-1}
 \bigl(\mathcal {\tilde G}(S)\bigr)^{-1} \mathcal Z^\vee(t)
 \bigl(\mathcal {\tilde G}(T)\bigr)^{-1},
 \end{align}
 where $\xi$ is from \eqref{eq:zeta}.
 \end{corollary}
 \begin{proof} Rearrange the terms in \eqref{eq:mainRes}.
 \end{proof}
 \begin{corollary} \label{cor:BG}
 For the algebra $\mathcal O_q$,
 \begin{align*}
 \lbrack \mathcal {\tilde G}(s), B(t) \rbrack = 0.
 \end{align*}
 \end{corollary}
 \begin{proof} The generating function $\mathcal {\tilde G}(s)$ commutes with each factor on the right in \eqref{eq:BZ}.
 \end{proof}
 \begin{corollary} \label{cor:BGc}
 For the algebra $\mathcal O_q$,
 \begin{align*}
 \lbrack \mathcal {\tilde G}_{k+1},  B_{n \delta} \rbrack = 0 \qquad \qquad k,n \in \mathbb N.
 \end{align*}
 \end{corollary}
 \begin{proof} By Corollary 
 \ref{cor:BG}.
 \end{proof}
 \begin{corollary} \label{cor:perm} The generating function
  $\mathcal Z^\vee (t)$
 \noindent is equal to each of
 \begin{align*}
 \xi \mathcal {\tilde G}(S) B(t) \mathcal {\tilde G}(T),
 \qquad \quad
 \xi B(t) \mathcal {\tilde G}(S) \mathcal {\tilde G}(T),
 \qquad \quad
 \xi \mathcal {\tilde G}(S) \mathcal {\tilde G}(T) B(t),
 \\
 \xi \mathcal {\tilde G}(T) B(t) \mathcal {\tilde G}(S),
 \qquad \quad
 \xi B(t) \mathcal {\tilde G}(T) \mathcal {\tilde G}(S),
 \qquad \quad
\xi \mathcal {\tilde G}(T) \mathcal {\tilde G}(S) B(t).
 \end{align*}
 \end{corollary}
 \begin{proof} Evaluate \eqref{eq:mainRes}
 using
  Corollary \ref{cor:BGc} and the equation on the right in \eqref{eq:3p10}.
 \end{proof}

 \section{The algebra homomorphism $\gamma: \mathcal O_q \to O_q$}
 
 In this section we construct a surjective algebra homomorphism  $\gamma: \mathcal O_q \to O_q$. We describe $\gamma$ in various ways.
We apply $\gamma$ to the alternating generators of $\mathcal O_q$, and obtain some elements called
 the alternating generators of $O_q$.
 \medskip
 
 \noindent Recall from  Proposition \ref{prop:three}  the algebra isomorphism $\phi:  O_q \otimes \mathbb F \lbrack z_1, z_2, \ldots \rbrack \to \mathcal O_q$.
\medskip

 \noindent We will be discussing the algebra homomorphism  $\theta: \mathbb F \lbrack z_1, z_2,\ldots \rbrack \to \mathbb F$ from Definition \ref{def:theta}.

 \begin{lemma} \label{lem:gam} There exists a unique algebra homomorphism $\gamma: \mathcal O_q \to O_q$ that makes the following diagram commute:
 
\begin{equation*}
{\begin{CD}
\mathcal O_q @<\phi << O_q \otimes \mathbb F \lbrack z_1, z_2, \ldots \rbrack
              \\
         @V \gamma VV                   @VV {\rm id}\otimes \theta V \\
         O_q @>>x \mapsto x \otimes 1>
                                 O_q \otimes \mathbb F 
                        \end{CD}}  	                          		    
\end{equation*}

 \end{lemma}
 \begin{proof} Concerning existence, define the map $\gamma$ to be the composition
 \begin{align*}
 \gamma: \quad 
 {\begin{CD}
\mathcal O_q @>>\phi^{-1}  > O_q \otimes \mathbb F \lbrack z_1, z_2, \ldots \rbrack  @>> {\rm id}\otimes \theta > O_q \otimes \mathbb F @>>x\otimes 1  \mapsto x >
                                 O_q.
                        \end{CD}}  	
                        \end{align*}
In this composition, each factor is an algebra homomorphism. Therefore $\gamma$ is an algebra homomorphism. By construction, $\gamma$
makes the diagram commute. We have shown that $\gamma$ exists. By construction $\gamma$ is unique.
                        \end{proof}
\begin{lemma} \label{lem:gamact} The algebra homomorphism $\gamma: \mathcal O_q \to O_q$ sends
\begin{align*}                     
 \mathcal W_0 \mapsto W_0,
 \qquad \quad
  \mathcal W_1 \mapsto W_1,
 \qquad \quad
 \mathcal Z^\vee_n \mapsto 0, \qquad n \geq 1.
 \end{align*}
\noindent Moreover, $\gamma$    is surjective.
\end{lemma}   
\begin{proof} Chase the $\mathcal O_q$-generators $\mathcal W_0$, $\mathcal W_1$, $\lbrace \mathcal Z^\vee_n \rbrace_{n=1}^\infty$ around the diagram in
Lemma \ref{lem:gam}, using   Proposition \ref{prop:three} and Definition \ref{def:theta}.
The map $\gamma$ is surjective since $W_0$, $W_1$ generate $O_q$.
\end{proof}
\noindent Next we describe how  $\gamma$ is related to the algebra homomorphism $\imath: O_q \to \mathcal O_q$ from Lemma \ref{lem:iota}.
\begin{lemma}
\label{lem:Ig}
The composition
 \begin{align*}
 {\begin{CD}
O_q @>> \imath  > \mathcal O_q  @>> \gamma > O_q 
                        \end{CD}}  	
                        \end{align*}
is equal to the identity map on $O_q$.
\end{lemma}
\begin{proof} By Lemmas \ref{lem:iota}, \ref{lem:gamact} the  given composition is an algebra homomorphism that fixes the $O_q$-generators $W_0$ and $W_1$.
\end{proof}

\begin{lemma} \label{lem:morecom}
The following diagrams commute:
\begin{equation*}
{\begin{CD}
\mathcal O_q @>\gamma  >> O_q
              \\
         @V \sigma VV                   @VV \sigma V \\
         \mathcal O_q @>>\gamma>
                                 O_q
                        \end{CD}}  	
         \qquad \qquad                
    {\begin{CD}
\mathcal O_q @>\gamma  >> O_q
              \\
         @V \dagger VV                   @VV \dagger V \\
        \mathcal O_q @>>\gamma >
                                  O_q
                        \end{CD}}  	     \qquad \qquad                
   {\begin{CD}
\mathcal O_q @>\gamma  >>O_q
              \\
         @V \tau VV                   @VV \tau V \\
        \mathcal O_q @>>\gamma >
                                 O_q
                        \end{CD}}  	                                       		    
\end{equation*}
\end{lemma}
\begin{proof} Chase the $\mathcal O_q$-generators $\mathcal W_0$, $\mathcal W_1$, $\lbrace \mathcal Z^\vee_n \rbrace_{n=1}^\infty$ around each diagram,
using Lemmas \ref{lem:aut}, \ref{lem:antiaut}, \ref{lem:autAc}, \ref{lem:antiautAc}
and 
Definitions \ref{def:tauA}, \ref{def:tauAc}
along with
Lemmas \ref{lem:Zfix}, \ref{lem:gamact}.
\end{proof}

\begin{definition} \label{def:altg} \rm
By an {\it alternating generator} of $O_q$ we mean the $\gamma$-image of an alternating generator for $\mathcal O_q$.
Our notation for an alternating generator of $O_q$ is given in the table below. For $k \in \mathbb N$,
\bigskip

\centerline{
\begin{tabular}[t]{c|cccc}
  $u$ & $\mathcal W_{-k}$ & $\mathcal W_{k+1}$ & $\mathcal G_{k+1}$ & $\mathcal {\tilde G}_{k+1}$
   \\
\hline
$\gamma(u)$ & $W_{-k}$  & $W_{k+1}$ & $G_{k+1}$  & $\tilde G_{k+1}$
\\
	       \end{tabular}}
\end{definition}	   
  
\noindent For notational convenience, define  
\begin{align}
G_0 = -(q-q^{-1})\lbrack 2 \rbrack^2_q, \qquad \qquad 
\tilde G_0 = -(q-q^{-1}) \lbrack 2 \rbrack^2_q.
\label{eq:GG02}
\end{align}

\begin{theorem} \label{lem:4gen} The alternating generators of $O_q$ satisfy the following relations.
For $k,\ell \in \mathbb N$,
\begin{align}
&
 \lbrack  W_0,W_{k+1}\rbrack= 
\lbrack W_{-k}, W_{1}\rbrack=
({{\tilde G}}_{k+1} -  G_{k+1})/(q+q^{-1}),
\label{eq:3p1c}
\\
&
\lbrack W_0, G_{k+1}\rbrack_q= 
\lbrack \tilde G_{k+1}, W_{0}\rbrack_q= 
\rho  W_{-k-1}-\rho 
 W_{k+1},
\label{eq:3p2c}
\\
&
\lbrack  G_{k+1},  W_{1}\rbrack_q= 
\lbrack W_{1}, {{\tilde G}}_{k+1}\rbrack_q= 
\rho  W_{k+2}-\rho 
 W_{-k},
\label{eq:3p3c}
\\
&
\lbrack W_{-k},  W_{-\ell}\rbrack=0,  \qquad 
\lbrack W_{k+1}, W_{\ell+1}\rbrack= 0,
\label{eq:3p4c}
\\
&
\lbrack W_{-k}, W_{\ell+1}\rbrack+
\lbrack  W_{k+1}, W_{-\ell}\rbrack= 0,
\label{eq:3p5c}
\\
&
\lbrack W_{-k}, G_{\ell+1}\rbrack+
\lbrack G_{k+1},  W_{-\ell}\rbrack= 0,
\label{eq:3p6c}
\\
&
\lbrack W_{-k}, {\tilde G}_{\ell+1}\rbrack+
\lbrack {\tilde G}_{k+1}, W_{-\ell}\rbrack= 0,
\label{eq:3p7c}
\\
&
\lbrack W_{k+1},  G_{\ell+1}\rbrack+
\lbrack  G_{k+1},  W_{\ell+1}\rbrack= 0,
\label{eq:3p8c}
\\
&
\lbrack W_{k+1}, {\tilde G}_{\ell+1}\rbrack+
\lbrack {\tilde G}_{k+1}, W_{\ell+1}\rbrack= 0,
\label{eq:3p9c}
\\
&
\lbrack G_{k+1}, G_{\ell+1}\rbrack=0,
\qquad 
\lbrack {\tilde G}_{k+1}, {\tilde G}_{\ell+1}\rbrack= 0,
\label{eq:3p10c}
\\
&
\lbrack {\tilde G}_{k+1}, G_{\ell+1}\rbrack+
\lbrack G_{k+1}, {\tilde G}_{\ell+1}\rbrack= 0.
\label{eq:3p11c}
\end{align}
We are using the notation  $\rho = -(q^2-q^{-2})^2$. 
  \end{theorem}         
  \begin{proof} Apply $\gamma$ to everything in  \eqref{eq:3p1}--\eqref{eq:3p11}, and evaluate the results using Definition  \ref{def:altg}.
  \end{proof}     
       \noindent Next, we describe how $\sigma$, $\dagger$, $\tau$ act on the alternating generators of $O_q$.
       
\begin{lemma}
\label{lem:autOc} The automorphism $\sigma$ of $O_q$ sends
\begin{align*}
W_{-k} \mapsto W_{k+1}, \qquad
 W_{k+1} \mapsto W_{-k}, \qquad
G_{k+1} \mapsto {\tilde G}_{k+1}, \qquad
{\tilde G}_{k+1} \mapsto  G_{k+1}
\end{align*}
 for $k \in \mathbb N$. 
\end{lemma}
\begin{proof} By Lemmas \ref{lem:autAc}, \ref{lem:morecom} and Definition \ref{def:altg}.
\end{proof}

\begin{lemma}\label{lem:antiautOc} The antiautomorphism $\dagger$ of $O_q$ sends
\begin{align*}
W_{-k} \mapsto W_{-k}, \qquad
 W_{k+1} \mapsto W_{k+1}, \qquad
G_{k+1} \mapsto {\tilde G}_{k+1}, \qquad
{\tilde G}_{k+1} \mapsto  G_{k+1}
\end{align*}
for $k \in \mathbb N$. 
\end{lemma}
\begin{proof}
By Lemmas \ref{lem:antiautAc}, \ref{lem:morecom} and Definition \ref{def:altg}.
\end{proof}

\begin{lemma}\label{def:tauOc} 
The  antiautomorphism $\tau$ of $O_q$ sends
\begin{align*}
W_{-k} \mapsto W_{k+1}, \qquad
W_{k+1} \mapsto W_{-k}, \qquad
G_{k+1} \mapsto G_{k+1}, \qquad
{\tilde G}_{k+1} \mapsto {\tilde G}_{k+1}
\end{align*}
\noindent for $k \in \mathbb N$. 
\end{lemma}
\begin{proof}  By
Definition \ref{def:tauAc} along with Lemma \ref{lem:morecom} and Definition \ref{def:altg}.
\end{proof}

\noindent Next, we describe the kernel of $\gamma$ in several ways.

\begin{proposition} The following are the same:
\begin{enumerate}
\item[\rm (i)] 
the kernel of $\gamma$;
\item[\rm (ii)]  the 2-sided ideal of $\mathcal O_q$
generated by $\lbrace \mathcal Z^\vee_n \rbrace_{n=1}^\infty$.
\end{enumerate}
\end{proposition}
\begin{proof}
We invoke the commuting diagram in Lemma \ref{lem:gam}.
Let $J$ denote the kernel of $\theta$.  By  Lemma  \ref{lem:idealZZ}, $J$ is the ideal of $\mathbb F\lbrack z_1, z_2, \ldots \rbrack$ generated by $\lbrace z_n \rbrace_{n=1}^\infty$. For 
the map  ${\rm id} \otimes \theta$ from the commuting diagram, the kernel 
 is $O_q \otimes J$  and this is the 2-sided ideal of $O_q \otimes  \mathbb F \lbrack z_1, z_2, \ldots \rbrack$ generated by $\lbrace 1 \otimes z_n \rbrace_{n=1}^\infty$.
 The algebra isomorphism $\phi $ sends $1\otimes z_n\mapsto \mathcal Z^\vee_n$ for $n\geq 1$.
The result follows from these comments and the commuting diagram in Lemma \ref{lem:gam}.
\end{proof}

\begin{proposition} The vector space $\mathcal O_q$
is the direct sum of the following:
\begin{enumerate}
\item[\rm (i)]  the kernel of $\gamma$;
\item[\rm (ii)]  the subalgebra $\langle \mathcal W_0, \mathcal W_1\rangle$ of $\mathcal O_q$.
\end{enumerate}
\end{proposition}
\begin{proof} By  Lemmas \ref{lem:iota}, \ref{lem:Ig}
and linear algebra.
\end{proof}

\section{More generating functions for $O_q$}

In this section we  use generating functions to describe the alternating generators of $O_q$.

\begin{definition}
\label{def:gf42}
\rm
We define some generating functions in the indeterminate $t$:
\begin{align*}
&W^-(t) = \sum_{n \in \mathbb N}  W_{-n} t^n,
\qquad \qquad  
W^+(t) = \sum_{n \in \mathbb N}  W_{n+1} t^n,
\\
&G(t) = \sum_{n \in \mathbb N} G_n t^n,
\qquad \qquad \qquad 
 {\tilde G}(t) = \sum_{n \in \mathbb N} {\tilde G}_n t^n.
\end{align*}
\end{definition}
\noindent Observe that
\begin{align*}
W^-(0) = W_0, \qquad
W^+(0) = W_1, \qquad
G(0) = -(q-q^{-1}) \lbrack 2 \rbrack^2_q, \qquad
{\tilde G}(0) = -(q-q^{-1}) \lbrack 2 \rbrack^2_q.
\end{align*}

\noindent  We now give the relations \eqref{eq:3p1c}--\eqref{eq:3p11c}  in terms of  generating functions. 

\begin{lemma} \label{lem:ad2} 
 For the algebra $O_q$ we have
\begin{align}
& \label{eq:3pp1c}
\lbrack W_0, W^+(t) \rbrack = \lbrack  W^-(t), W_1 \rbrack = t^{-1}( {\tilde G}(t)- G(t))/(q+q^{-1}),
\\
& \label{eq:3pp2c}
\lbrack  W_0, G(t) \rbrack_q = \lbrack {\tilde G}(t),  W_0 \rbrack_q = \rho  W^-(t)-\rho t  W^+(t),
\\
&\label{eq:3pp3c}
\lbrack  G(t), W_1 \rbrack_q = \lbrack  W_1,  {\tilde G}(t) \rbrack_q = \rho W^+(t) -\rho t  W^-(t),
\\
&\label{eq:3pp4c}
\lbrack  W^-(s), W^-(t) \rbrack = 0, 
\qquad 
\lbrack W^+(s),  W^+(t) \rbrack = 0,
\\ \label{eq:3pp5c}
&\lbrack   W^-(s),  W^+(t) \rbrack 
+
\lbrack  W^+(s), W^-(t) \rbrack = 0,
\\ \label{eq:3pp6c}
&s \lbrack  W^-(s), G(t) \rbrack 
+
t \lbrack  G(s),  W^-(t) \rbrack = 0,
\\ \label{eq:3pp7c}
&s \lbrack  W^-(s),  {\tilde G}(t) \rbrack 
+
t \lbrack  {\tilde G}(s),  W^-(t) \rbrack = 0,
\\ \label{eq:3pp8c}
&s \lbrack    W^+(s),  G(t) \rbrack
+
t \lbrack   G(s), W^+(t) \rbrack = 0,
\\ \label{eq:3pp9c}
&s \lbrack  W^+(s), {\tilde G}(t) \rbrack
+
t \lbrack {\tilde G}(s),  W^+(t) \rbrack = 0,
\\ \label{eq:3pp10c}
&\lbrack   G(s), G(t) \rbrack = 0, 
\qquad 
\lbrack  {\tilde G}(s),   {\tilde G}(t) \rbrack = 0,
\\ \label{eq:3pp11c}
&\lbrack  {\tilde G}(s), G(t) \rbrack +
\lbrack    G(s), {\tilde G}(t) \rbrack = 0.
\end{align}
\end{lemma}
\begin{proof} Apply $\gamma$ to everything in
Lemma \ref{lem:ad}.
\end{proof}

\noindent So far, it appears that $O_q$ resembles $\mathcal O_q$. However, this resemblance extends only so far. The next result holds for $O_q$ but not $\mathcal O_q$.

\begin{proposition}
\label{prop:ConjT} For the algebra $O_q$,
\begin{align}
B(t) \tilde G(T) \tilde G(S) = -q^{-1}(q-q^{-1})^3 \lbrack 2 \rbrack^4_q.
\label{eq:ConjT}
\end{align}
\end{proposition}
\begin{proof} For notational convenience, we identify $O_q$ with  $\langle \mathcal W_0, \mathcal W_1 \rangle$ via the
map $\imath$ from Lemma \ref{lem:iota}.
By Corollary  \ref{cor:perm} the following holds for $\mathcal O_q$:
 \begin{align} 
 \label{eq:prep}
 B(t) \mathcal {\tilde G}(T) \mathcal {\tilde G}(S) =
 \xi^{-1} \mathcal Z^\vee(t).
 \end{align}
 We apply $\gamma$ to each side of \eqref{eq:prep}. By Lemma
 \ref{lem:Ig}  the map $\gamma$ fixes everything in $O_q$, so $\gamma $ fixes $B(t)$.
 By Definition \ref{def:altg} the map
  $\gamma$ sends  $\mathcal {\tilde G}(S) \mapsto  {\tilde G}(S) $ and
  $\mathcal {\tilde G}(T) \mapsto  {\tilde G}(T) $.
 By construction  $\mathcal Z^\vee(t) = \sum_{n \in \mathbb N} \mathcal Z^\vee_n t^n$ and $\mathcal Z^\vee_0=1$.
By Lemma \ref{lem:gamact} the map $\gamma$ sends $\mathcal Z^\vee_n \mapsto 0$ for $n\geq 1$. Therefore $\gamma$ sends $\mathcal Z^\vee(t)\mapsto 1$. The result follows in view
of \eqref{eq:zeta}.
\end{proof}

\begin{theorem}\label{thm:ct} The conjecture {\rm \cite[Conjecture 6.2]{conj}}
is true.
\end{theorem}
\begin{proof}  The equation \eqref{eq:ConjT}
is the same as \cite[Eqn.~(41)]{conj}, in view of \eqref{eq:ST}.
Consequently
 \cite[Conjecture 6.2]{conj} is true by the discussion above \cite[Eqn.~(41)]{conj}.
 \end{proof}

\noindent The article \cite{conj} discusses the meaning of \eqref{eq:ConjT}. Below we summarize a few points.

\begin{lemma} \label{lem:detail}
{\rm (See \cite[Eqn.~(43)]{conj}.)}  For the algebra $O_q$ the elements $\lbrace \tilde G_n \rbrace_{n=1}^\infty$ and $\lbrace B_{n\delta} \rbrace_{n=1}^\infty$ are recursively obtained from
each other as follows.
For $n\geq 1$,
\begin{align}
\label{eq:recBG}
0 = \lbrack n \rbrack_q B_{n\delta}  {\tilde G}_0 
+
\sum_{\stackrel{\scriptstyle j+k+2\ell+1=n,}{\scriptstyle j,k,\ell\geq 0}}
(-1)^\ell 
\binom{k+\ell}{\ell}
\lbrack 2n-j \rbrack_q
\lbrack 2 \rbrack^{k+1}_q
B_{j\delta}  {\tilde G}_{k+1}.
\end{align}
\end{lemma}
\noindent See \cite[Appendix~A]{conj} for a detailed discussion of \eqref{eq:recBG}. Next we give a consequence of \eqref{eq:recBG}.

\begin{lemma}
\label{lem:GBcom} {\rm (See \cite[Lemma~4.10]{conj}.)}
For the algebra $ O_q$ and $n\geq 1$,
\begin{enumerate}
\item[\rm (i)]
 the element $\tilde G_n$ is a polynomial
of total degree $n$ in $B_\delta, B_{2\delta}, \ldots, B_{n\delta}$,
where we view $B_{k\delta}$ as having degree $k$ for $1\leq k \leq n$.
For this polynomial the constant term is 0, and the coefficient of
$B_{n\delta}$ is $-q \lbrack n \rbrack_q \lbrack 2n \rbrack^{-1}_q \lbrack 2 \rbrack^{2-n}_q$;
\item[\rm (ii)]
the element $B_{n\delta}$ is a polynomial
of total degree $n$ in $\tilde G_1, \tilde G_2, \ldots, \tilde G_n$,
where we view $\tilde G_k$ as having degree $k$ for $1 \leq k \leq n$.
For this polynomial the constant term is 0, and the coefficient of
$\tilde G_n$ is  $-q^{-1} \lbrack n \rbrack^{-1}_q \lbrack 2n \rbrack_q \lbrack 2 \rbrack^{n-2}_q$.
\end{enumerate}
\end{lemma}
\begin{proof} By Lemma \ref{lem:detail} and induction on $n$.
\end{proof}

\section{The algebra isomorphism $\varphi: \mathcal O_q \to O_q \otimes \mathbb F \lbrack z_1, z_2,\ldots \rbrack$}

\noindent In this section we introduce an algebra isomorphism $\varphi: \mathcal O_q \to O_q \otimes \mathbb F \lbrack z_1, z_2,\ldots \rbrack$. As we will see, $\varphi$ is closely related to the inverse of the map $\phi$ from Proposition \ref{prop:three}.

\begin{lemma}
\label{lem:varphi}
There exists an algebra homomorphism $\varphi:
\mathcal O_q \to O_q \otimes 
\mathbb F \lbrack z_1, z_2,\ldots\rbrack$ that sends
\begin{align*}
\mathcal W_{-n} &\mapsto \sum_{k=0}^n W_{k-n} \otimes z_k,
\quad \qquad \qquad 
\mathcal W_{n+1} \mapsto \sum_{k=0}^n W_{n+1-k} \otimes z_k,
\\
\mathcal G_{n} &\mapsto \sum_{k=0}^n G_{n-k} \otimes z_k,
\quad \qquad \qquad
\mathcal {\tilde G}_{n} \mapsto \sum_{k=0}^n \tilde G_{n-k} \otimes z_k
\end{align*}
for $n \in \mathbb N$. In particular $\varphi$ sends
\begin{align}
\mathcal W_0 \mapsto W_0 \otimes 1,
\qquad \qquad 
\mathcal W_1 \mapsto W_1 \otimes 1.
\label{eq:vphiW0W1}
\end{align}
\end{lemma}
\begin{proof} One checks using  \eqref{eq:3p1c}--\eqref{eq:3p11c} that for
 the alternating generators of $\mathcal O_q$ their $\varphi$-image candidates satisfy the defining relations 
 for $\mathcal O_q$ given in \eqref{eq:3p1}--\eqref{eq:3p11}.
\end{proof}
\noindent Recall the algebra homomorphism $\imath: O_q \to \mathcal O_q$ from Lemma \ref{lem:iota}.
\begin{lemma} \label{lem:gam3} The following diagram commutes:
 
\begin{equation*}
{\begin{CD}
O_q @>x \mapsto x \otimes 1 >>  O_q \otimes \mathbb F \lbrack z_1, z_2, \ldots \rbrack  
              \\
         @V \imath VV                   @VV {\rm id} V \\
        \mathcal O_q @>>\varphi > O_q \otimes \mathbb F \lbrack z_1, z_2, \ldots \rbrack                            
                        \end{CD}}  	                          		    
\end{equation*}
\end{lemma}
\begin{proof} Chase the $O_q$-generators $W_0$, $W_1$ around the diagram, using Lemma \ref{lem:iota} and
\eqref{eq:vphiW0W1}.
\end{proof}

\noindent Our next goal is to show that $\varphi$ is an algebra isomorphism. Recall the central elements $\lbrace \mathcal Z^\vee_n \rbrace_{n \in \mathbb N}$ for $\mathcal O_q$, from Definition \ref{def:Zn}.
In Appendix A we describe an algebra isomorphism $\vee:  \mathbb F \lbrack z_1, z_2, \ldots \rbrack \to  \mathbb F \lbrack z_1, z_2, \ldots \rbrack$.
We will be discussing the images $\lbrace z^\vee_n \rbrace_{n\in \mathbb N}$ of  $\lbrace z_n \rbrace_{n\in \mathbb N}$.

\begin{lemma} \label{lem:vpZ}
The map $\varphi$ sends
$ \mathcal Z^\vee_n \mapsto 1\otimes z^\vee_n$ for $n \in \mathbb N$.
\end{lemma}
\begin{proof} For notational convenience, we identify $O_q$ with $\langle \mathcal W_0, \mathcal W_1 \rangle$ via the map $\imath$ from Lemma \ref{lem:iota}.
We will work with generating functions. It suffices to show that $\varphi$ sends $\mathcal Z^\vee (t) \mapsto 1 \otimes Z^\vee (t)$, where $Z^\vee(t)$ is from 
Definition \ref{def:Zcheck}.
By Corollary  \ref{cor:perm}, 
\begin{align*}
\label{eq:zcheckvphi}
\mathcal Z^\vee(t) = \xi B(t) \mathcal {\tilde G}(T) \mathcal {\tilde G}(S).
\end{align*}
By Lemma \ref{lem:gam3},  $\varphi$ sends $B_{n\delta} \mapsto B_{n\delta}\otimes 1$ for $n \in \mathbb N$. Therefore $\varphi$ sends $B(t) \mapsto B(t) \otimes 1$.
By Lemma \ref{lem:varphi},  $\varphi$ sends
$\mathcal {\tilde G}(t) \mapsto {\tilde G}(t) \otimes Z(t)$, where $Z(t)$ is from Definition \ref{lem:poly}. Therefore $\varphi$ sends
$\mathcal {\tilde G}(T) \mapsto {\tilde G}(T) \otimes Z(T)$ and $\mathcal {\tilde G}(S) \mapsto {\tilde G}(S) \otimes Z(S)$.
 By these comments, $\varphi$ sends
\begin{align*}
\mathcal Z^\vee(t) \mapsto \xi B(t) {\tilde G}(T) {\tilde G}(S) \otimes Z(S)Z(T).
\end{align*}
By Proposition \ref{prop:ConjT}, $\xi B(t) {\tilde G}(T) {\tilde G}(S)=1$. By Proposition \ref{def:zcheck},
 $Z(S)Z(T)=Z^\vee(t)$.
Consequently $\varphi $ sends $\mathcal Z^\vee(t) \mapsto 1 \otimes Z^\vee(t)$. The result follows.
\end{proof}

\begin{proposition}\label{prop:vphiP} The following diagram commutes:

\begin{equation*}
{\begin{CD}
O_q \otimes \mathbb F \lbrack z_1, z_2, \ldots \rbrack@>\phi  >> \mathcal O_q 
              \\
         @V {\rm id}\otimes \vee VV                   @VV {\rm id} V\\
     O_q \otimes \mathbb F \lbrack z_1, z_2, \ldots \rbrack   @<<\varphi <
                                 \mathcal O_q
                        \end{CD}}  	                          		    
\end{equation*}

\end{proposition}
\begin{proof} Chase the generators $W_0 \otimes 1$, $W_1 \otimes 1$, $\lbrace 1 \otimes z_n \rbrace_{n=1}^\infty$
around the diagram, using Proposition
\ref{prop:three} along with \eqref{eq:vphiW0W1}
and Lemma \ref{lem:vpZ}.
\end{proof}

\begin{theorem} 
\label{prop:vpIso}
The map $\varphi$ from Lemma \ref{lem:varphi} is an algebra isomorphism.
\end{theorem}
\begin{proof} By Proposition \ref{prop:vphiP}, and since the maps $\phi$ and $\vee$ are algebra isomorphisms.
\end{proof}

\begin{lemma} \label{lem:vphOnto}
 The algebra isomorphism $\varphi$ sends $\langle \mathcal  W_0, \mathcal W_1\rangle $ onto $O_q \otimes 1$.
 \end{lemma}
 \begin{proof} By Lemma \ref{lem:gam3}.     
 \end{proof}

\begin{lemma} \label{lem:vphiZact} The algebra isomorphism $\varphi$ sends $\mathcal Z$ onto $1 \otimes \mathbb F \lbrack z_1, z_2, \ldots \rbrack$.
\end{lemma}
\begin{proof} By Propositions
\ref{prop:three}, \ref{prop:vphiP}.        
\end{proof}

\begin{definition}\label{def:NewZ} \rm
For $n \in \mathbb N$ let $\mathcal Z_n$ denote the unique element in $\mathcal Z$ that  $\varphi$ sends to $1 \otimes z_n$.
Note that $\mathcal Z_0=1$.
\end{definition}

\begin{lemma} \label{lem:NewZai} The elements $\lbrace \mathcal Z_n \rbrace_{n=1}^\infty$ are algebraically independent and generate $\mathcal Z$. 
\end{lemma}
\begin{proof} 
 The elements $\lbrace z_n \rbrace_{n=1}^\infty$
are algebraically independent and generate $\mathbb F \lbrack z_1, z_2,\ldots \rbrack$. So the elements
 $\lbrace 1\otimes z_n \rbrace_{n=1}^\infty$
are algebraically independent and generate $1\otimes \mathbb F \lbrack z_1, z_2,\ldots \rbrack$.
The result follows in view of Lemma \ref{lem:vphiZact} and Definition \ref{def:NewZ}.
\end{proof}

\noindent We have seen that $\varphi: \mathcal O_q \to O_q \otimes 
\mathbb F \lbrack z_1, z_2,\ldots\rbrack$  is an algebra isomorphism that sends
\begin{align*}
\mathcal Z_n \mapsto 1 \otimes z_n, \qquad \qquad \mathcal Z^\vee_n \mapsto 1 \otimes z^\vee_n, \qquad \qquad n\in \mathbb N.
\end{align*}

\noindent In the next two results, we clarify how $\lbrace \mathcal Z_n \rbrace_{n=1}^\infty$ and $\lbrace \mathcal Z^\vee_n \rbrace_{n=1}^\infty$ are related.

\begin{lemma}
\label{lem:ZNewpoly} For the algebra $\mathcal O_q$ and $n\geq 1$,
\begin{enumerate}
\item[\rm (i)]
 the element $\mathcal Z^\vee_n$ is a polynomial
of total degree $n$ in $\mathcal Z_1, \mathcal Z_2, \ldots, \mathcal Z_n$,
where we view  $\mathcal Z_k$ as having degree $k$ for $1\leq k \leq n$.
For this polynomial the constant term is 0, and the coefficient of
$\mathcal Z_n$ is $(q+q^{-1})^n (q^n+q^{-n})$;
\item[\rm (ii)]
the element $\mathcal Z_n$ is a polynomial
of total degree $n$ in $\mathcal Z^\vee_1, \mathcal Z^\vee_2, \ldots, \mathcal Z^\vee_n$,
where we view $\mathcal Z^\vee_k$ as having degree $k$ for $1\leq k\leq n$.
For this polynomial the constant term is 0, and the coefficient of
$\mathcal Z^\vee_n$ is $(q+q^{-1})^{-n} (q^n+q^{-n})^{-1}$.
\end{enumerate}
\end{lemma}
\begin{proof} By the comment below Lemma \ref{lem:NewZai}, along with
 Lemmas 
\ref{lem:zvpoly}, \ref{lem:zpoly}.
\end{proof}

\begin{lemma} \label{lem:ideaZZ} For the algebra $\mathcal O_q$ the following are the same:
\begin{enumerate}
\item[\rm (i)] the 2-sided ideal  generated by  $\lbrace \mathcal Z_n \rbrace_{n=1}^\infty$;
\item[\rm (ii)] the 2-sided ideal  generated by  $\lbrace \mathcal Z^\vee_n \rbrace_{n=1}^\infty$.
\end{enumerate}
\end{lemma}
\begin{proof} By Lemma \ref{lem:ZNewpoly}.
\end{proof}

\begin{lemma}
\label{lem:pov}
For the algebra $\mathcal O_q$ and  $n \in \mathbb N$,
\begin{align*}
\mathcal W_{-n} &= \sum_{k=0}^n W_{k-n} \mathcal Z_k,
\quad \qquad \qquad 
\mathcal W_{n+1} = \sum_{k=0}^n W_{n+1-k} \mathcal Z_k,
\\
\mathcal G_{n} &=\sum_{k=0}^n G_{n-k} \mathcal Z_k,
\quad \qquad \qquad
\mathcal {\tilde G}_{n} = \sum_{k=0}^n \tilde G_{n-k} \mathcal Z_k.
\end{align*}
\noindent In the above lines we identify $O_q$ with $\langle \mathcal W_0, \mathcal W_1\rangle$ via $\imath$.
\end{lemma}
\begin{proof} By Lemmas  \ref{lem:varphi}, \ref{lem:gam3} and Definition \ref{def:NewZ}.
\end{proof}

\begin{definition} \label{def:onemoreGF} \rm Define the generating function
\begin{align}
\label{eq:Z2GF}
\mathcal Z(t) = \sum_{n\in \mathbb N} \mathcal Z_n t^n.
\end{align}
\end{definition}

\begin{theorem} \label{lem:ZST} We have
\begin{align}
\mathcal Z^\vee (t) = \mathcal Z(S) \mathcal Z(T),
\label{eq:ZST}
\end{align}
where $S, T$ are from  \eqref{eq:ST}.
\end{theorem}
\begin{proof} By the comment below Lemma \ref{lem:NewZai}, along with Proposition \ref{def:zcheck}.
\end{proof} 

\begin{proposition} \label{prop:ZST} For the algebra $\mathcal O_q$ we have
\begin{align*}
\mathcal W^-(t) &= W^-(t) \mathcal Z(t), \qquad \qquad \,\mathcal W^+(t) = W^+(t) \mathcal Z(t), 
\\
\mathcal  G(t) &= G(t) \mathcal Z(t), \qquad \qquad \qquad \mathcal {\tilde G}(t) = {\tilde G}(t) \mathcal Z(t).
\end{align*}
\noindent In the above lines we identify $O_q$ with $\langle \mathcal W_0, \mathcal W_1\rangle$ via $\imath$.
\end{proposition}
\begin{proof} By 
Lemma \ref{lem:pov} and Definitions \ref{def:gf42},  \ref{def:onemoreGF}.
\end{proof}

\noindent The following result is a variation on Proposition \ref{lem:Zdiag}.
\begin{proposition} 
\label{lem:Zdiag2} 
The following diagrams commute:

\begin{align*}
&{\begin{CD}
\mathcal O_q  @>\varphi  >> O_q \otimes \mathbb F \lbrack z_1, z_2, \ldots \rbrack               \\
         @V \sigma VV                   @VV  \sigma\otimes {\rm id}V \\
     \mathcal O_q   @>>\varphi >
                                 O_q \otimes \mathbb F \lbrack z_1, z_2, \ldots \rbrack  
                        \end{CD}}  	
         \qquad       \qquad     
{\begin{CD}
\mathcal O_q  @>\varphi  >> O_q \otimes \mathbb F \lbrack z_1, z_2, \ldots \rbrack               \\
         @V \dagger VV                   @VV  \dagger\otimes {\rm id}V \\
     \mathcal O_q   @>>\varphi >
                                 O_q \otimes \mathbb F \lbrack z_1, z_2, \ldots \rbrack  
                        \end{CD}}  	  
                        \\          
& {\begin{CD}
\mathcal O_q  @>\varphi  >> O_q \otimes \mathbb F \lbrack z_1, z_2, \ldots \rbrack               \\
         @V \tau VV                   @VV  \tau\otimes {\rm id}V \\
     \mathcal O_q   @>>\varphi >
                                 O_q \otimes \mathbb F \lbrack z_1, z_2, \ldots \rbrack  
                        \end{CD}}  	      		    
\end{align*}
\end{proposition}
\begin{proof} Chase the $\mathcal O_q$-generators $\mathcal W_0$, $\mathcal W_1$, $\lbrace \mathcal Z^\vee_n \rbrace_{n=1}^\infty$ around each diagram using
Lemmas \ref{lem:aut}, \ref{lem:antiaut}, \ref{lem:autAc}, \ref{lem:antiautAc}
and 
Definitions \ref{def:tauA}, \ref{def:tauAc}
along with
Lemmas \ref{lem:Zfix}, \ref{lem:varphi}, \ref{lem:vpZ}.
\end{proof}

\noindent The following result is a variation on Lemma \ref{lem:gam}.

 \begin{lemma} \label{lem:gam2} The following diagram commutes:
 
\begin{equation*}
{\begin{CD}
\mathcal O_q @>\varphi >> O_q \otimes \mathbb F \lbrack z_1, z_2, \ldots \rbrack
              \\
         @V \gamma VV                   @VV {\rm id}\otimes \theta V \\
         O_q @>>x \mapsto x \otimes 1>
                                 O_q \otimes \mathbb F 
                        \end{CD}}  	                          		    
\end{equation*}
 
 \end{lemma}
\begin{proof} Chase the $\mathcal O_q$-generators $\mathcal W_0$, $\mathcal W_1$, $\lbrace \mathcal Z^\vee_n \rbrace_{n=1}^\infty$ around the diagram,
using Lemmas \ref{lem:gamact},  \ref{lem:idealZZ} along with  \eqref{eq:vphiW0W1}
and the comment below Lemma \ref{lem:NewZai}.
\end{proof}

\section{The algebra homomorphism $\eta: \mathcal O_q \to \mathbb F\lbrack z_1, z_2, \ldots \rbrack$}

\noindent In this section, we introduce a surjective algebra homomorphism
 $\eta: \mathcal O_q \to \mathbb F\lbrack z_1, z_2, \ldots \rbrack$. We use $\eta$ to illuminate how $\mathcal O_q$ is related to $O_q$. We describe the kernel of $\eta$
 in several ways.

\begin{lemma} \label{lem:eta}
There exists an algebra homomorphism $\eta : \mathcal O_q \to \mathbb F \lbrack z_1, z_2, \ldots \rbrack$ that sends
\begin{align}
\label{eq:assign}
\mathcal W_{-n} \mapsto 0, \qquad 
\mathcal W_{n+1} \mapsto 0, \qquad 
\mathcal G_{n} \mapsto \mathcal G_0 z_{n}, \qquad 
\mathcal {\tilde G}_{n} \mapsto \mathcal {\tilde G}_0z_{n}
\end{align}
for $n \in \mathbb N$.
Moreover $\eta$ is surjective.
\end{lemma}
\begin{proof}  The algebra homomorphism $\eta$ exists because the defining relations \eqref{eq:3p1}--\eqref{eq:3p11} for $\mathcal O_q$
hold if we make the assignments \eqref{eq:assign}. The map $\eta$ is surjective since $\lbrace z_n \rbrace_{n=1}^\infty$ generate
$\mathbb F \lbrack z_1, z_2, \ldots \rbrack$.
\end{proof}

\noindent Shortly  we will describe how $\eta$ is related to the isomorphism $\varphi$ from Theorem \ref{prop:vpIso}.
We will use the following result.
\begin{lemma} \label{lem:etaZ} The map $\eta$ sends
$\mathcal Z^\vee_n \mapsto z^\vee_n$ for $n \in \mathbb N$.
\end{lemma}
\begin{proof} We will work with generating functions. It suffices to show that $\eta$ sends $\mathcal Z^\vee(t) \mapsto Z^\vee(t)$. By Definition \ref{def:ZV}
we have $\mathcal Z^\vee(t) = \lbrack 2 \rbrack^{-2}_q \Psi(t)$, where $\Psi(t)$ is from Definition
 \ref{lem:zV1}.
By Lemma \ref{lem:eta}, the map $\eta$ sends 
\begin{align*}
\mathcal W^-(t) \mapsto 0, \qquad 
\mathcal W^+(t) \mapsto 0, \qquad 
\mathcal G(t) \mapsto \mathcal G_0 Z(t), \qquad 
\mathcal {\tilde G}(t) \mapsto \mathcal {\tilde G}_0 Z(t).
\end{align*}
By this and Definition \ref{lem:zV1},  $\eta$ sends
\begin{align*}
\Psi(t) \mapsto (q^2-q^{-2})^{-2} \mathcal G_0 \mathcal {\tilde G}_0 Z(S)Z(T).
\end{align*}
By \eqref{eq:GG0} we have  $(q^2-q^{-2})^{-2} \mathcal G_0 \mathcal {\tilde G}_0= \lbrack 2 \rbrack^2_q$, and by Proposition \ref{def:zcheck}  we have  $Z^\vee(t)=Z(S)Z(T) $.
\noindent By these comments, $\eta$ sends $\mathcal Z^\vee(t) \mapsto Z^\vee(t)$.
\end{proof}

\noindent Recall the algebra homomorphism $\vartheta: O_q \to \mathbb F$ from Lemma \ref{lem:vth}.

\begin{proposition}\label{prop:etaD} The following diagram commutes:

\begin{equation*}
{\begin{CD}
\mathcal O_q @>\varphi >> O_q \otimes \mathbb F \lbrack z_1, z_2, \ldots \rbrack
              \\
         @V \eta VV                   @VV \vartheta \otimes {\rm id}V \\
        \mathbb F \lbrack z_1, z_2, \ldots \rbrack @>>x \mapsto 1 \otimes x>
                                 \mathbb F \otimes \mathbb F \lbrack z_1, z_2, \ldots \rbrack
                        \end{CD}}  	                          		    
\end{equation*}

\end{proposition}
\begin{proof} Chase the $\mathcal O_q$-generators $\mathcal W_0$, $\mathcal W_1$, $\lbrace \mathcal Z^\vee_n\rbrace_{n=1}^\infty$ around the diagram, using \eqref{eq:vphiW0W1}
and Lemmas \ref{lem:vth}, \ref{lem:vpZ}, \ref{lem:eta}, \ref{lem:etaZ}.
\end{proof}

\begin{corollary} \label{cor:vth3}
The map $\eta$ sends $\mathcal Z_n \mapsto z_n $ for $n \in \mathbb N$.
\end{corollary}
\begin{proof}  By the comment below Lemma  \ref{lem:NewZai} along with
Proposition \ref{prop:etaD}.
\end{proof}

\noindent Next, we describe the kernel of $\eta$ in several ways.

\begin{proposition} The following are the same:
\begin{enumerate}
\item[\rm (i)] the kernel of $\eta$;
\item[\rm (ii)] the 2-sided ideal of $\mathcal O_q$
generated by $\mathcal W_0$, $\mathcal W_1$.
\end{enumerate}
\end{proposition}
\begin{proof} We invoke the commuting diagram in  Proposition \ref{prop:etaD}.
Let $K$ denote the kernel of the algebra homomorphism $\vartheta: O_q \to \mathbb F$. By Lemma \ref{lem:vth}, $K$ is the 2-sided ideal of $O_q$ generated by $W_0$, $W_1$. For 
the map $\vartheta\otimes {\rm id}$ from the commuting diagram, the kernel 
 is $K \otimes \mathbb F \lbrack z_1, z_2, \ldots \rbrack$  and this is the 2-sided ideal of $O_q \otimes  \mathbb F \lbrack z_1, z_2, \ldots \rbrack$ generated by $W_0\otimes 1$, $W_1 \otimes 1$.
 The algebra isomorphism $\varphi $ sends $\mathcal W_0 \mapsto W_0\otimes 1$ and $\mathcal W_1 \mapsto W_1\otimes 1$.
The result follows from these comments and the commuting diagram in Proposition \ref{prop:etaD}.
\end{proof}

\begin{proposition} 
The vector space $\mathcal O_q$ is the direct
sum of the following:
\begin{enumerate}
\item[\rm (i)] the center $\mathcal Z$ of $\mathcal O_q$;
\item[\rm (ii)] the kernel of $\eta$. 
\end{enumerate}
\end{proposition}
\begin{proof} By Lemma  \ref{lem:vphiZact} we have
 $\varphi(\mathcal Z)=1 \otimes \mathbb F \lbrack z_1, z_2, \ldots \rbrack$.
By this and the commuting diagram in  Proposition \ref{prop:etaD}, the restriction of $\eta$ to $\mathcal Z$ gives a bijection $ \mathcal Z \to
\mathbb F \lbrack z_1, z_2, \ldots \rbrack$.
The result follows from this
 and linear algebra.
\end{proof}

\section{The algebra homomorphism $\vartheta: O_q \to \mathbb F$, revisited}

For the sake of completeness, we show how the algebra homomorphism  $\vartheta: O_q \to \mathbb F$ from Lemma \ref{lem:vth} acts on the alternating generators of $O_q$.

\begin{lemma} The map $\vartheta$ sends
\begin{align*}
W_{-n}\mapsto 0, \qquad
W_{n+1} \mapsto 0, \qquad
G_{n+1} \mapsto 0, \qquad
{\tilde G}_{n+1} \mapsto 0
\end{align*}
for $n \in \mathbb N$.
\end{lemma}
\begin{proof} We show that $\vartheta(W_{-n})=0$. Without loss of generality, we may assume that $\vartheta(W_{k-n})=0$ for $1 \leq k \leq n$. We chase $\mathcal W_{-n}$ around the diagram in Proposition \ref{prop:etaD}, using 
Lemmas  \ref{lem:varphi}, \ref{lem:eta}. For one path in the diagram the outcome is 0, and for the other path in the diagram the outcome is $\vartheta(W_{-n})\otimes 1$. Therefore $\vartheta(W_{-n})=0$.
The remaining assertions are obtained in a similar manner.
\end{proof}

\section{Directions for future research}

In this section we give some suggestions for future research.
\begin{conjecture} \label{prob1}
\rm
Lemma \ref{prop:damiani} 
remains valid if we remove the assumption that $q$ is transcendental, and require only that $q$ is not a root of unity.
\end{conjecture}

\noindent The following conjecture is a variation on \cite[Conjecture~1]{basBel}.

\begin{conjecture} \label{prob2}\rm
A PBW basis for $O_q$ is obtained by the elements
\begin{align}
\label{eq:omega}
\lbrace W_{-i} \rbrace_{i \in \mathbb N}, \qquad 
\lbrace \tilde G_{j+1} \rbrace_{j\in \mathbb N}, \qquad  
\lbrace W_{k+1} \rbrace_{k\in \mathbb N}
\end{align}
in any linear order $<$ that satisfies one of {\rm (i)--(vi)} below:
\begin{enumerate}
\item[\rm (i)]  $W_{-i} <  \tilde G_{j+1} < W_{k+1}$ for $i,j,k\in \mathbb N$;
\item[\rm (ii)]  $W_{k+1} < \tilde G_{j+1} < W_{-i}$ for $i,j,k\in \mathbb N$;
\item[\rm (iii)]  $W_{k+1} < W_{-i} < \tilde G_{j+1}$ for $i,j,k\in \mathbb N$;
\item[\rm (iv)]  $W_{-i} < W_{k+1} < \tilde G_{j+1}$ for $i,j,k\in \mathbb N$;
\item[\rm (v)]  $\tilde G_{j+1} < W_{k+1} < W_{-i}$ for $i,j,k\in \mathbb N$;
\item[\rm (vi)]  $\tilde G_{j+1} < W_{-i} < W_{k+1}$ for $i,j,k\in \mathbb N$.
\end{enumerate}
\end{conjecture}

\noindent We motivate the next problem with  some comments.
\medskip

\noindent Let $\mathbb Z_4 =  {\mathbb Z} /4 \mathbb Z$
denote the cyclic group of order $4$.

\begin{definition} {\rm (See \cite[Definition~5.1]{pospart}.)}
\rm
\label{def:boxqV1M}
Let $\square_q$ denote the algebra with
generators $\lbrace x_i\rbrace_{i\in \mathbb Z_4}$
and the following relations. For $i \in \mathbb Z_4$,
\begin{align*}
&
\quad \qquad \qquad 
\frac{q x_i x_{i+1}-q^{-1}x_{i+1}x_i}{q-q^{-1}} = 1,
\\
&
x^3_i x_{i+2} - \lbrack 3 \rbrack_q x^2_i x_{i+2} x_i +
\lbrack 3 \rbrack_q x_i x_{i+2} x^2_i -x_{i+2} x^3_i = 0.
\end{align*}
\end{definition}

\begin{lemma} {\rm (See \cite[Proposition~5.6, Theorem~10.33]{pospart}.)}
\label{lem:square}
Pick nonzero $a,b \in \mathbb F$.
Then there exists an injective algebra homomorphism
$\sharp: O_q \to 
\square_q$ that sends
\begin{align*}
W_0 \mapsto a x_{0}+ a^{-1} x_{1},
\qquad \qquad
W_1 \mapsto  b x_{2}+ b^{-1} x_{3}.
\end{align*}
\end{lemma}
\begin{problem}\rm
\label{prob3}
Find the image of each alternating generator of $O_q$, under the homomorphism $\sharp$ from Lemma \ref{lem:square}.
\end{problem}

\begin{problem}\rm Let $V$ denote a finite-dimensional irreducible $O_q$-module on which each of $W_0$, $W_1$ is diagonalizable.
Consider the four types of alternating generators for $O_q$:
\begin{align*}
\lbrace W_{-k} \rbrace_{k \in \mathbb N}, \qquad
\lbrace W_{k+1} \rbrace_{k \in \mathbb N}, \qquad
\lbrace G_{k+1} \rbrace_{k \in \mathbb N}, \qquad
\lbrace \tilde G_{k+1} \rbrace_{k \in \mathbb N}.
\end{align*}
For each type the alternating generators mutually commute; find their eigenvalues and common eigenvectors in $V$.
Perhaps start by assuming that $V$ has shape $(1,2,1)$ in the sense of Vidar \cite{vidar}.
\end{problem}

\begin{conjecture}\rm For $x \in O_q$ the following are equivalent:
\begin{enumerate}
\item[\rm (i)] $x$ commutes with $W_{-k}$ for $k \in \mathbb N$;
\item[\rm (ii)] $x$ is contained in the subalgebra of $O_q$ generated by 
$\lbrace W_{-k} \rbrace_{k \in \mathbb N}$.
\end{enumerate}
\end{conjecture}

\begin{conjecture}\rm For $x \in O_q$ the following are equivalent:
\begin{enumerate}
\item[\rm (i)] $x$ commutes with $ {\tilde G}_{k+1}$ for $k \in \mathbb N$;
\item[\rm (ii)] $x$ is contained in the subalgebra of $O_q$ generated by 
$\lbrace  {\tilde G}_{k+1} \rbrace_{k \in \mathbb N}$.
\end{enumerate}
\end{conjecture}

\section{Appendix A: the algebra $\mathbb F \lbrack z_1, z_2, \ldots \rbrack$}

\noindent In this appendix, we explain some features of the polynomial algebra $\mathbb F \lbrack z_1, z_2, \ldots \rbrack$ that are used
in the main body of the paper.  Recall the notation $z_0=1$.
\medskip

\begin{definition}\label{lem:poly} \rm Define the generating function
\begin{align}
Z(t) = \sum_{n \in \mathbb N} z_n t^n.
\label{eq:Zpoly}
\end{align}
\end{definition}

 \begin{lemma}\label{lem:warm}  {\rm (See \cite[Lemma~4.5]{conj}.)}
 We have
 \begin{align}
Z  \biggl(\frac{q+q^{-1}}{t+t^{-1}}\biggr) = \sum_{n \in \mathbb N} z^\downarrow_n \lbrack 2 \rbrack^n_q t^n,
\label{lem:Zbig}
\end{align} 
  where $z^\downarrow_0=1$ and
\begin{align}
 z^\downarrow_n = \sum_{\ell=0}^{\lfloor (n-1) /2 \rfloor} (-1)^\ell \binom{n-1-\ell}{\ell} \lbrack 2 \rbrack^{-2\ell}_q z_{n-2\ell}, \qquad \quad n\geq 1.
\label{eq:zdsum}
\end{align}
\end{lemma}
\begin{proof} This is routinely checked.
\end{proof}

\begin{example}\rm In the table below, we display $z^\downarrow_n$  for $1 \leq n \leq 9$.
\bigskip

\centerline{
\begin{tabular}[t]{c|c}
  $n$ & $z^\downarrow_{n}$
   \\
\hline
$1$ & $z_{1} $ \\
$2$ & $z_{2} $ \\
$3$ & $z_{3} - \lbrack 2 \rbrack^{-2}_q z_1$ \\
$4$ & $z_{4} -2\lbrack 2 \rbrack^{-2}_q z_2$ \\
$5$ & $z_{5} -3  \lbrack 2 \rbrack^{-2}_q z_3+  \lbrack 2 \rbrack^{-4}_q z_1$ \\
$6$ & $z_{6} -4 \lbrack 2 \rbrack^{-2}_q z_4+ 3 \lbrack 2 \rbrack^{-4}_qz_2 $ \\
$7$ & $z_{7} -5  \lbrack 2 \rbrack^{-2}_q z_5 + 6 \lbrack 2 \rbrack^{-4}_q z_3-  \lbrack 2 \rbrack^{-6}_q z_1$ \\
$8$ & $z_{8} -6 \lbrack 2 \rbrack^{-2}_q z_6+ 10 \lbrack 2 \rbrack^{-4}_qz_4- 4\lbrack 2 \rbrack^{-6}_qz_2 $ \\
$9$ & $z_{9} -7  \lbrack 2 \rbrack^{-2}_q z_7+ 15 \lbrack 2 \rbrack^{-4}_qz_5 - 10\lbrack 2 \rbrack^{-6}_q z_3 + \lbrack 2 \rbrack^{-8}_q z_1$ 
\end{tabular}}
\bigskip

\end{example}

\noindent Recall the functions $S$, $T$  from \eqref{eq:ST}.
\begin{lemma} \label{lem:ramp} 
We have
\begin{align*}
Z(S) &= \sum_{n \in \mathbb N} z^\downarrow_n q^{-n} \lbrack 2 \rbrack^n_q t^n,
\qquad \qquad 
Z(T) = \sum_{n \in \mathbb N} z^\downarrow_n q^{n} \lbrack 2 \rbrack^n_q t^n.
\end{align*}
\end{lemma}
\begin{proof} By Lemma \ref{lem:warm}  with $t$ replaced by $q^{-1}t$ or $qt$.
\end{proof}

\begin{definition}\label{def:Zcheck} \rm For $n \in \mathbb N$ define
\begin{align}
z^\vee_n = \lbrack 2 \rbrack^n_q \sum_{k=0}^n q^{n-2k} z^\downarrow_k z^\downarrow_{n-k}.
\label{eq:Zcheck}
\end{align}
Note that $z^\vee_0=1$. Further define
\begin{align}
Z^\vee(t) = \sum_{n \in \mathbb N} z^\vee_n t^n.
\label{eq:Zvpoly}
\end{align}
\end{definition}

\begin{proposition} \label{def:zcheck}  We have
\begin{align}
\label{eq:zvee}
Z^\vee(t) = Z(S)Z(T).
\end{align}
\end{proposition}
\begin{proof} By Lemma \ref{lem:ramp} and Definition \ref{def:Zcheck}.
\end{proof}

\begin{example}\label{ex:zvee} \rm
We have
\begin{align*}
z^\vee_1 &= (q+q^{-1})^2z_1,
\\
z^\vee_2 &= (q+q^{-1})^2(q^2+q^{-2})z_2 + (q+q^{-1})^2z^2_1,
\\
z^\vee_3 &= (q+q^{-1})^3(q^3+q^{-3})z_3 + (q+q^{-1})^4z_1z_2-(q+q^{-1})(q^3+q^{-3})z_1.
\end{align*}
\end{example}

\begin{lemma}
\label{lem:zvpoly}
For $n\geq 1$ the element $z^\vee_n$ is a polynomial
of total degree $n$ in $z_1, z_2, \ldots, z_n$,
where we view $z_k$ as having degree $k$ for $1 \leq k \leq n$.
For this polynomial the constant term is 0, and the coefficient of
$z_n$ is $(q+q^{-1})^n (q^n+q^{-n})$.
\end{lemma}
\begin{proof} By
\eqref{eq:zdsum} and \eqref{eq:Zcheck}.
\end{proof}

\noindent The following comments will help us interpret Lemma \ref{lem:zvpoly}.
We describe a basis for the vector space $\mathbb F \lbrack z_1, z_2, \ldots \rbrack$.
For $n \in \mathbb N$, a {\it partition of $n$} is a sequence $\lambda = \lbrace \lambda_i \rbrace_{i=1}^\infty$
of natural numbers such that $\lambda_i \geq \lambda_{i+1}$ for $i\geq 1$ and $n=\sum_{i=1}^\infty \lambda_i$.
The set $\Lambda_n$ consists of the partitions of $n$. Define $\Lambda = \cup_{n \in \mathbb N} \Lambda_n$. 
For $\lambda \in \Lambda$ define $z_\lambda = \prod_{i=1}^\infty z_{\lambda_i}$. The elements 
$\lbrace z_\lambda\rbrace_{\lambda \in \Lambda}$ form a basis for the vector space $\mathbb F \lbrack z_1, z_2, \ldots \rbrack$.
\medskip

\noindent
Next we construct a grading for the algebra $\mathbb F \lbrack z_1, z_2, \ldots \rbrack$.
For notational convenience abbreviate  $P=\mathbb F \lbrack z_1, z_2, \ldots \rbrack$.
  For $n \in \mathbb N$ let $P_n$ denote the subspace of
$P$ with basis $\lbrace z_\lambda \rbrace_{\lambda \in \Lambda_n}$. For example $ P_0 = \mathbb F 1$.
The sum $P = \sum_{n\in \mathbb N} P_n$ is direct. Moreover $P_r P_s\subseteq P_{r+s}$
for $r,s\in \mathbb N$. By these comments 
the subspaces $\lbrace P_n \rbrace_{n\in \mathbb N}$ form
a grading of the algebra $P$. For $n \in \mathbb N$, the dimension of $P_n$ is equal to the number of partitions of $n$. Observe that
\begin{align}
P_n = \mathbb F z_n + \sum_{k=1}^{n-1} {\rm Span} (P_k P_{n-k}), \qquad \qquad n\geq 1.
\label{eq:Pdec}
\end{align}
\noindent 
\medskip

\noindent We now interpret Lemma \ref{lem:zvpoly} in light of the above comments.

\begin{lemma} \label{lem:interpP}
For $n \geq 1$ we have $z^\vee_n \in \sum_{k=1}^n P_k$ and
\begin{align}
z^\vee_n - (q+q^{-1})^n (q^n+q^{-n}) z_n \in \sum_{k=1}^{n-1} P_k+ \sum_{k=1}^{n-1} {\rm Span}(P_k P_{n-k}).
\label{lem:prod}
\end{align}
\end{lemma} 
\begin{proof} By Lemma \ref{lem:zvpoly}.
\end{proof}

\begin{lemma} \label{lem:mapV}  There exists an algebra homomorphism $\vee: \mathbb F \lbrack z_1, z_2, \ldots \rbrack\to\mathbb F \lbrack z_1, z_2, \ldots \rbrack$ that sends
$z_n \mapsto z^\vee_n$ for $n\geq 1$. 
\end{lemma} 
\begin{proof} Since $\lbrace z_n \rbrace_{n=1}^\infty$ are algebraically independent and generate $\mathbb F \lbrack z_1, z_2, \ldots \rbrack$.
\end{proof}

\noindent Our next goal is to show that the map $\vee$ from Lemma \ref{lem:mapV} is an algebra isomorphism.

\begin{definition} \label{lem:zv}\rm  For $z \in \mathbb F \lbrack z_1, z_2, \ldots \rbrack$ let $z^\vee$ denote the image of $z$ under the
homomorphism $\vee$ from Lemma \ref{lem:mapV}.
\end{definition}

\begin{lemma} \label{lem:need} We have $P^\vee_0=P_0$ and for $n\in \mathbb N$,
\begin{align}
\label{eq:ss}
P^\vee_1 + P^\vee_2 + \cdots + P^\vee_n  = P_1 + P_2 + \cdots + P_n.
\end{align}
\end{lemma}
\begin{proof} We have $P^\vee_0=P_0$ since  $P_0=\mathbb F 1$ and $1^\vee=1$.
Next we prove \eqref{eq:ss} by induction on $n$.  For notational convenience, define $F_n = P_1+\cdots + P_n$.  We show that $F^\vee_n=F_n$.
This holds vacuosly for $n=0$, so assume that $n\geq 1$.
 By induction,  $F^\vee_\ell=F_\ell$ for $0 \leq \ell \leq n-1$.
Adjusting \eqref{eq:Pdec}, we obtain
\begin{align}
F_n=  F_{n-1} + \mathbb F z_n + \sum_{k=1}^{n-1} {\rm Span}(F_k F_{n-k}).
\label{eq:FF1}
\end{align}
\noindent Upon applying $\vee$ to each side of \eqref{eq:FF1},  
\begin{align}
F^\vee_n 
&=  F^\vee_{n-1} + \mathbb F z^\vee_n + \sum_{k=1}^{n-1} {\rm Span}(F^\vee_k F^\vee_{n-k}) \nonumber \\
&=  F_{n-1} + \mathbb F z^\vee_n + \sum_{k=1}^{n-1} {\rm Span}(F_k F_{n-k}). \label{eq:FF2}
\end{align}
By \eqref{lem:prod} and the construction,
\begin{align*}
z^\vee_n - (q+q^{-1})^n (q^n+q^{-n}) z_n \in F_{n-1} + \sum_{k=1}^{n-1} {\rm Span}(F_k F_{n-k}).
\end{align*}
Consequently,
the right-hand side of 
\eqref{eq:FF1} is equal to the right-hand side of
\eqref{eq:FF2}. 
By these comments $F^\vee_n = F_n$.
\end{proof}

\begin{proposition}
\label{prop:mapIso}
The map $\vee$ from Lemma \ref{lem:mapV}
is an algebra isomorphism.
\end{proposition}
\begin{proof}
By construction $\vee$ is an algebra homomorphism. We show that $\vee$ is a bijection.
 For $n \in \mathbb N$, the restriction of $\vee$ to $\sum_{k=0}^n P_k$ gives a bijection $\sum_{k=0}^n P_k \to \sum_{k=0}^n P_k$ by Lemma \ref{lem:need} and since $\sum_{k=0}^n P_k$ has finite dimension.
 We have  $\mathbb F \lbrack z_1, z_2, \ldots \rbrack=\sum_{k=0}^\infty P_k$.
 Therefore $\vee$ is a
 bijection.
\end{proof}

\begin{corollary}
\label{cor:homzzv}
The elements $\lbrace z^\vee_n\rbrace_{n=1}^\infty$
are algebraically independent and generate 
$\mathbb F \lbrack z_1, z_2,\ldots\rbrack$.
\end{corollary}
\begin{proof} The elements $\lbrace z_n \rbrace_{n=1}^\infty$ are algebraically independent and generate $\mathbb F \lbrack z_1, z_2, \ldots \rbrack$.
The result follows from this and Proposition \ref{prop:mapIso}.
\end{proof}

\noindent For $n\geq 1$, we now seek to express
$z_n$ as a polynomial in $z^\vee_1, z^\vee_2,\ldots, z^\vee_n$. For $n=1,2,3$ we obtain the following from Example \ref{ex:zvee}.

\begin{example}
\label{ex:znsolve}
We have
\begin{align*}
z_1 &= \frac{z^\vee_1}{(q+q^{-1})^2},
\\
z_2 &= \frac{(q+q^{-1})^2 z^\vee_2-(z^\vee_1)^2}{(q+q^{-1})^4(q^2+q^{-2})},
\\
z_3 &= \frac{(q+q^{-1})^2 (q^2+q^{-2})z^\vee_3-(q+q^{-1})^2z^\vee_1 z^\vee_2+
(z^\vee_1)^3+(q+q^{-1})(q^2+q^{-2})(q^3+q^{-3})z^\vee_1}{(q+q^{-1})^5(q^2+q^{-2})(q^3+q^{-3})}.
\end{align*}
\end{example}

\begin{lemma}
\label{lem:zpoly}
For $n\geq 1$ the element $z_n$ is a polynomial
of total degree $n$ in $z^\vee_1, z^\vee_2, \ldots, z^\vee_n$,
where we view $z^\vee_k$ as having degree $k$ for $1 \leq k \leq n$.
For this polynomial the constant coefficient is $0$, and the coefficient of
$z^\vee_n$ is $(q+q^{-1})^{-n}(q^n+q^{-n})^{-1}$.
\end{lemma}
\begin{proof} This follows from Lemmas \ref{lem:interpP}, \ref{lem:need}. Alternatively,
 in Lemma \ref{lem:zvpoly} solve for $z_n$ and evaluate the result using induction on $n$.
\end{proof}

\begin{definition}\label{def:theta} \rm
 Let $\theta: \mathbb F \lbrack z_1, z_2,\ldots \rbrack \to \mathbb F$ denote the algebra homomorphism that sends $z_n \mapsto 0$ for $n\geq 1$. 
\end{definition}

\begin{lemma} \label{lem:idealZZ} For the algebra $\mathbb F \lbrack z_1, z_2, \ldots\rbrack$ the following {\rm (i)--(iv)} are the same:
\begin{enumerate}
\item[\rm (i)] the kernel of  $\theta$;
\item[\rm (ii)] the sum $\sum_{n=1}^\infty P_n$, where $\lbrace P_n \rbrace_{n\in \mathbb N}$ is the grading of  $\mathbb F \lbrack z_1, z_2, \ldots\rbrack$ from below Lemma \ref{lem:zvpoly};
\item[\rm (iii)] the ideal generated by $\lbrace z_n \rbrace_{n=1}^\infty$;
\item[\rm (iv)] the ideal generated by $\lbrace z^\vee_n \rbrace_{n=1}^\infty$.
\end{enumerate}
\end{lemma} 
\begin{proof} It is routine to check that (i)--(iii) are the same; denote this common value by $J$. Comparing (iii), (iv) we see that (iv) is equal to $J^\vee$. Applying $\vee$ to (ii) and using Lemma 
 \ref{lem:need} we obtain
 $J^\vee=J$.
\end{proof}

\section{Acknowledgements}
The author thanks Pascal Baseilhac and Travis Scrimshaw for helpful conversations about the $q$-Onsager algebra and its alternating central extension.


%

\bigskip

\noindent Paul Terwilliger \hfil\break
\noindent Department of Mathematics \hfil\break
\noindent University of Wisconsin \hfil\break
\noindent 480 Lincoln Drive \hfil\break
\noindent Madison, WI 53706-1388 USA \hfil\break
\noindent email: {\tt terwilli@math.wisc.edu }\hfil\break

\end{document}